\documentclass[a4paper,reqno,11pt]{amsart}

\usepackage{amsmath,amssymb,amsthm,booktabs, cases}
\usepackage{mathrsfs}
\usepackage{amsmath,amssymb,amsthm,enumerate,framed,graphicx,array,color,multirow,mathrsfs,amsrefs}

\usepackage{enumitem}
\setlist[itemize]{parsep=2.5pt}
\setlist[enumerate]{parsep=2.5pt}

\usepackage[utf8]{inputenc}
\usepackage[colorlinks,citecolor=blue,urlcolor=blue]{hyperref}

\usepackage{bm}
\usepackage{euscript}
\usepackage{graphicx}
\usepackage{comment}
\usepackage{multicol}
\usepackage[usenames,dvipsnames,svgnames,table]{xcolor}

\usepackage{a4wide}

\usepackage{mathrsfs}
\usepackage{euscript}

\numberwithin{equation}{section} \theoremstyle{plain}
\newtheorem{theorem}{Theorem}[section]

\newtheorem{corollary}[theorem]{Corollary}

\newtheorem{definition}[theorem]{Definition}

\newtheorem{lemma}[theorem]{Lemma}

\newtheorem{proposition}[theorem]{Proposition}

\theoremstyle{remark}
\newtheorem{remark}[theorem]{Remark}

\theoremstyle{remark}

\allowdisplaybreaks

\allowdisplaybreaks[4]
\numberwithin{equation}{section}

\allowdisplaybreaks

\allowdisplaybreaks[4]
\numberwithin{equation}{section}

\def\bC{\mathbb{C}}
\def\bN{\mathbb{N}}
\def\R{\mathbb{R}}
\def\E{\mathbb{E}}
\def\bP{\mathbb{P}}

\def\cH{\mathcal{H}}

\def\cB{\mathcal{B}}

\def\cF{\mathcal{F}}

\def\bB{\mathbf B}

\def\1{\mathbf 1}

\def\e{\varepsilon}

\begin{document}

\title[Hitting probabilities of Gaussian fields]{Hitting probabilities of Gaussian random fields and collision of eigenvalues of random matrices}

\author[C. Lee]{Cheuk Yin Lee}
\address{Institut de math\'ematiques, \'Ecole polytechnique f\'ed\'erale de Lausanne,
Station 8, CH-1015 Lausanne, Switzerland }
\email{cheuk.lee@epfl.ch}

\author[J. Song]{Jian Song} 
\address{Research Center for Mathematics and Interdisciplinary Sciences, Shandong University, Qingdao, Shandong, 266237, China;
    and School of Mathematics, Shandong University, Jinan, Shandong, 250100, China}
    \email{txjsong@sdu.edu.cn}
    
    \author[Y. Xiao]{Yimin Xiao}
    \address{Department of Statistics and Probability, Michigan State University, A-413 Wells Hall, East Lansing,
MI 48824, U.S.A.}
\email{xiaoy@msu.edu}

\author[W. Yuan]{Wangjun Yuan}
\address{Department of Mathematics and Statistics,  University of Ottawa, Canada}
\email{ywangjun@connect.hku.hk}

%
%
%
%
%

\subjclass[2010]{60B20, ~60G15,~60G22}

\keywords{Gaussian random fields, fractional Brownian motion, hitting probabilities, Gaussian orthogonal ensemble, 
Gaussian unitary ensemble, eigenvalues}

\date{}

\begin{abstract}  
 Let $X= \{X(t), t \in \R^N\}$ be a centered Gaussian random field with values in $\R^d$ satisfying certain conditions and let
$F \subset \R^d$ be a Borel set.  In our main theorem, we provide a sufficient condition for $F$ to 
be polar for $X$, i.e. $\mathbb P\big( X(t) \in F \hbox{ for some }  t \in \R^N\big) = 0$, 
which improves significantly the main result in Dalang et al 
\cite{dalang2017polarity}, where the case of $F$ being a singleton was considered. 
 We provide a variety of  examples of Gaussian random field for which  our result is applicable. Moreover, by using our main theorem, 
 we solve a problem on the existence of collisions of the eigenvalues of 
random matrices with Gaussian random field entries that was left open in Jaramillo and Nualart 
\cite{jaramillo2020collision} and Song et al \cite{song2021collision}.
\end{abstract}

\maketitle

{
\hypersetup{linkcolor=black}
 \tableofcontents 
}

\section{Introduction}

This paper is motivated by a problem on the existence of collision of the eigenvalues of a random matrix
with Gaussian random field entries that has been left open by Jaramillo and Nualart \cite{jaramillo2020collision}   
and Song et al \cite{song2021collision}. We start by describing briefly some history and existing results on the aforementioned problem.
 In the celebrated work \cite{Dyson1962}, Dyson introduced independent Ornstein-Uhlenbeck processes to a Hermitian matrix as its entries and showed that the system of eigenvalue processes models the so-called time-dependent Coulomb gas. Later on, it was shown that the eigenvalue processes never collide almost surely (see, e.g., \cite{rs93}).	For a symmetric matrix with independent Brownian motion entries, its eigenvalues do not collide for almost all trajectories and satisfy a system of the It\^o stochastic differential equations with
non-smooth diffusion coefficients. The process formed by the ordered eigenvalues of the symmetric Brownian motion matrix  
is now known as Dyson's non-colliding Brownian motion (see, 
e.g., \cite{mckean, rs93, anderson2010} for more information). If the  matrix entries are fractional Brownian motions 
with Hurst parameter $H$,   Nualart and P\'erez-Abreu \cite{nualart2014eigenvalue} proved that, when 
$H\in(\frac12,1)$, the eigenvalue processes do not collide by using the stochastic calculus with respect to Young 
integrals.  When $H\in(0, \frac12)$, Jaramillo and Nualart \cite{jaramillo2020collision} identified the collision 
probability of the eigenvalues with the hitting probability of Gaussian fields and, as a consequence, obtained 
a sufficient condition and a necessary condition for the positivity of the collision probability. The result of  
\cite{jaramillo2020collision} was recently extended to the collision probability of multiple eigenvalues  by 
Song et al \cite{song2021collision} who also obtained the Hausdorff dimension of the set of collision times.  
The methodology based on hitting probabilities used in \cite{jaramillo2020collision,song2021collision}, 
which first appeared in McKean \cite[Section 4.9]{mckean} in the proof of the non-collision property for Dyson's 
Brownian motion, can deal with the collision problem for the eigenvalues of random matrices with  
more general Gaussian random field entries including fractional Brownian motion with multidimensional 
indices and the Brownian sheet. 
One of the key ingredients in \cite{jaramillo2020collision,song2021collision} is the result in \cite{BLX2009} on  
hitting probability of Gaussian random fields. While \cite{BLX2009} is useful for determining whether $k$ 
eigenvalues of a real symmetric random matrix with Gaussian random field entries may collide or not for 
the cases $\sum_{j=1}^N\frac1{H_j} > (k+2)(k-1)/2$ and $\sum_{j=1}^N\frac1{H_j} < (k+2)(k-1)/2$ under 
the setting of \cite{jaramillo2020collision,song2021collision} (see Section \ref{sec:rm} below), 
it does not provide any useful information when $\sum_{j=1}^N\frac1{H_j} = (k+2)(k-1)/2$, which is referred to 
as the critical dimension case for the collision problem. In this case, the problem on the existence of collision 
of the eigenvalues of a real symmetric (or complex Hermitian) random matrix with Gaussian random field 
entries has been left open by Jaramillo and Nualart \cite{jaramillo2020collision} and Song et al \cite{song2021collision}.  


In this paper, we solve this problem by first establishing a hitting probability result that is stronger than that in 
\cite{BLX2009} for a large class of Gaussian random fields. More specifically, let $X= \{X(t), t \in \R^N\}$ be 
a centered Gaussian random field with values in $\R^d$ (for brevity, $X$ is called an $(N, d)$-random field) 
that satisfies the general  assumptions in Dalang et al \cite{dalang2017polarity}.  We derive in Theorem 
\ref{Th:main} a sufficient condition for a Borel set $F \subset \R^d$ to be polar for $X$, i.e.,
$\mathbb P\big( X(t) \in F \hbox{ for some }  t \in \R^N\big) = 0$ in terms of a condition related to the upper
Minkowski dimension of $F$. This theorem improves significantly Theorem 2.6 in Dalang et al \cite{dalang2017polarity}, 
where the case of $F$ being a singleton was considered, and is applicable to solutions of stochastic partial 
differential equations (SPDEs). 


The method for proving Theorem \ref{Th:main} is based on a refined covering argument.  Compared with   
\cite{BLX2009} and other related references for hitting probabilities of Gaussian random fields and 
solutions to SPDEs such as \cite{dalang2007hitting, dalang2009hitting, dalang2013hitting, 
dalang2010criteria, dalang2015hitting, hinojosa2021anisotropic, Xiao2009sample},  the method 
for  constructing the covering sets in this paper is significantly different. 
In \cite{BLX2009} and the other references, the authors covered the inverse image $\{t \in I: X(t) \in F\}$, 
where $I \subset \R^N$ is a compact interval, by balls whose sizes are determined by $F$ and the largest 
global oscillation of $X$ on $I$. Consequently, these coverings are quite coarse and the covering argument 
fails if the dimension of $F$ is critical for the polarity problem for $X$ (e.g., $\dim F = d - \frac N H$ when $X$ is 
an $(N, d)$-fractional Brownian motion of index $H$). In the present paper, we construct a random covering 
for $\{t \in I: X(t) \in F\}$ by using balls whose sizes match the smallest local oscillation of $X$ with very large 
probability, see Proposition \ref{Prop:X} below and the proof of Theorem \ref{Th:main}. 
Our covering argument is originated from Talagrand \cite{Talagrand95, Talagrand98} and extends the method 
in \cite{dalang2017polarity}. 



The rest of this paper is organized as follows.  In Section \ref{sec:hitting}, we study the hitting probability of Gaussian 
random fields under the general setting of Dalang et al \cite{dalang2017polarity}. The main result is Theorem \ref{Th:main}, 
which provides a sufficient condition related to the upper Minkowski dimension for a Borel set $F \subset \R^d$ to be polar 
for $X$. In Section \ref{sec:examples}, we give some examples of Gaussian random fields that satisfy the conditions 
imposed in Section \ref{sec:hitting}. In particular, we show that  the solutions of the systems of linear stochastic heat and 
wave equations with a Gaussian noise that is white in time and colored in space satisfy the conditions of Theorem \ref{Th:main}. 
This allows us to strengthen the results in \cite{dalang2017polarity} and prove the polarity of a class of sets with critical 
dimension for the solutions of these SPDEs. 
In Section \ref{sec:rm}, we apply the main result Theorem \ref{Th:main} to study the collision problem for the eigenvalues 
of random matrices with Gaussian random field entries and prove that there is no collision of $k$ eigenvalues 
of the real symmetric random matrices when $\sum_{j=1}^N\frac1{H_j} = (k+2)(k-1)/2$. This solves a problem that 
was left open in \cite{jaramillo2020collision, song2021collision}.

\section{Hitting probabilities in critical dimension}\label{sec:hitting}

Let $X:=\{X(t) = (X_1(t), \ldots, X_d(t)), t\in \R^N\}$ be a centered continuous $\R^d$-valued Gaussian random field 
defined on a probability space $(\Omega, \cF, \bP)$. In this section, we study the hitting probabilities of $X$ in 
critical dimension in a general setting of Dalang et al \cite{dalang2017polarity}. 

First we recall the following definition of an $\R^d$-valued Gaussian noise on $\R_+$.
\begin{definition} \label{Def-noise}
	Let $\nu$ be a Borel measure on $\R_+$, and let $A \mapsto W(A)$ be a set function defined on $\cB (\R_+)$ 
	with values in $L^2(\Omega, \cF, \bP; \R^d)$ such that for each $A$, $W(A)$ is a centered normal random vector with 
	values in $\R^d$ and covariance matrix $\nu(A) I_d$. Assume that $W(A \cup B) = W(A) + W(B)$, and $W(A)$ 
	and $W(B)$ are independent whenever $A \cap B = \emptyset$. Then the set function $A \mapsto W(A)$ is called
	an $\R^d$-valued Gaussian noise with control measure $\nu$. 
\end{definition} 

As in \cite{dalang2017polarity}, we assume that the component processes $X_1, \ldots, X_d$ of the $(N, d)$-random field $X$ 
are i.i.d.\footnote{While the independence of the component processes of $X$ plays an important role in this paper, 
the condition for them to be identically distributed can be relaxed, see Section 4 for an example.}  
For finite constants $c_j < d_j$ ($j=1, \dots, N$),   let 
\[
I := \prod_{j=1}^N [c_j,d_j]
\] 
be a compact interval in $\R^N$. Denote $I^{(\epsilon)} := \prod_{j=1}^N (c_j-\epsilon,d_j+\epsilon)$. 

We impose the following assumption on the Gaussian random field $X$, which is the same as \cite[Assumption 2.1]{dalang2017polarity}.
\begin{enumerate}
\item[{\bf  (A1)}] There is a Gaussian random field $\{W(A,t): A \in \cB(\R_+), t \in \R^N\}$ and $\epsilon_0>0$ 
satisfying the following two conditions. 
\begin{enumerate}
\item[(a1)] For all $t \in I^{(\epsilon_0)}$, $A \mapsto W(A,t)$ is an $\R^d$-valued Gaussian noise with a control 
measure $\nu_t$ such that $W(\R_+,t) = X(t)$ and  when $A \cap B = \emptyset$, $W(A,\cdot)$ and $W(B,\cdot)$ are 
independent.
	
\item[(a2)] There exist constants $a_0 \ge 0$, $c_0 > 0$, 
$\gamma_j > 0$, $j = 1, \dots, N$, such that for all $a_0 \le a < b
 \le +\infty$ and all $s:=(s_1,\dots, s_N),t:=(t_1, \dots, t_N) \in I^{(\epsilon_0)}$,
		\begin{align*}
			\big\| W([a,b),s) - X(s) - W([a,b),t) + X(t) \big\|_{L^2}
			\le c_0 \left[ \sum_{j=1}^N a^{\gamma_j} |s_j - t_j| + b^{-1} \right],
		\end{align*}
		and
		\begin{align*}
			\big\| W([0,a_0),s) - W([0,a_0),t) \big\|_{L^2}
		\le c_0 \sum_{j=1}^N |s_j - t_j|,
		\end{align*}
		 where $\|Y\|_{L^2}:=\left(\E[|Y|^2]\right)^{1/2}=\left(\E[ Y_1^2+\dots+Y_d^2 ]\right)^{1/2}$  for a random vector 
		$Y:=(Y_1, \dots, Y_d)$.

	\end{enumerate}
\end{enumerate}

Denote
	\begin{equation}\label{e:alpha}
	\alpha_j = (1+\gamma_j)^{-1}, \ \  1 \le j \le N.
	\end{equation}
	   These parameters are useful for characterizing various properties of  the random field $X$. Define a metric on $\R^N$ by
\begin{align}\label{e:Delta}
	\Delta(s,t) = \sum_{j=1}^N |s_j - t_j|^{\alpha_j}, ~ s, t\in \R^N.
\end{align}
The following lemma is from \cite[Proposition 2.2]{dalang2017polarity}. It enables us to 
bound the canonical metric on $I^{(\epsilon_0)}$ induced by $\|X(s)-X(t)\|_{L^2}$ by using the metric $\Delta$. 
 
\begin{lemma}\label{Lem:distance}
	Under Assumption {\bf (A1)}, for all $s,t \in I^{(\epsilon_0)}$ with $\Delta(s,t) \le \min \{a_0^{-1}, 1\}$, we have
	\begin{align*}
		\big\| X(s) - X(t) \big\|_{L^2}
		\le 4c_0 \Delta(s,t).
	\end{align*}
\end{lemma}

 We further impose the following two assumptions on $X$, which are aslo stated in \cite[Assumption~2.4]{dalang2017polarity}. 
\begin{enumerate}
	\item[{\bf (A2)}] There exists a constant $d_0 > 0$, such that $\|X_i(t)\|_{L^2} \ge d_0$ for all $t \in I^{(\epsilon_0)}$ 
	and all $1 \le i \le N$.
	
	\item[{\bf (A3)}] There exists a constant ${\rho_0} > 0$ with the following property. For $t \in I$, there exist
	$t' = t'(t) \in I^{(\epsilon_0)}$, $\delta_j = \delta_j(t) \in (\alpha_j,1]$ for $1 \le i \le N$ (recalling that $\alpha_j$'s are 
	given in \eqref{e:alpha}), and $C = C(t) > 0$, such that
	\begin{align*}
		\Big| \E \left[ X_i(t') \big( X_i(s) - X_i(\bar{s}) \big) \right] \Big|
		\le C \sum_{j=1}^N |s_j - \bar{s}_j|^{\delta_j},
	\end{align*}
	for all $1 \le i \le N$ and all $s, \bar{s} \in I^{(\epsilon_0)}$ with $\max\{\Delta(t,s), \Delta(t,\bar{s})\} \le 2 {\rho_0}$.
\end{enumerate}
Notice that {\bf (A2)} is a non-degeneracy condition on $X$ and that {\bf (A3)} is a regularity condition  which yields 
better path regularity of $X$ than Lemma \ref{Lem:distance}.

We now introduce the following important parameter:
	\begin{align}\label{e:Q}
 Q = \sum_{j=1}^N \dfrac{1}{\alpha_j}.
	\end{align}
It follows from \cite{BLX2009,Xiao2009sample} that if Assumption {\bf (A1)} holds,  $d \ge Q$, and $F\subset \R^d$ 
has $(d-Q)$-dimensional Hausdorff measure 0, then $F$ is polar for $X$. However, if the $(d-Q)$-dimensional 
Hausdorff measure of $F$ is not 0 (this is always the case if $d = Q$ and $F \ne \emptyset$), it is in general not 
known whether $F$ is polar for $X$ or not. The special case of  $F= \{x\}$ when $d=Q$ was solved by Dalang et 
al \cite{dalang2017polarity}. (For completeness, we mention that if $d < Q$, then for every $x \in \R^d$, $X^{-1}(x) \ne 
\emptyset$ with positive probability, see \cite[Theorem 7.1]{Xiao2009sample}.)

The following is the main result of this section which provides a sufficient condition on  $F\subset \R^d$ such that    
$X^{-1}(F) \cap I=\emptyset$ a.s. This result improves Theorem 2.6 in Dalang et al \cite{dalang2017polarity} and 
the results in \cite{BLX2009,Xiao2009sample}. For a general Gaussian random field $X$  (except the Brownian motion 
and  the Brownian sheet which were completely solved by Kakutani  \cite{kakutani1944} and by Khoshnevisan and 
Shi \cite{khoshnevisan1999brownian}, respectively) the condition 
(\ref{Con:F}) on $F$ is the weakest general condition so far for the polarity of $F$.


\begin{theorem} \label{Th:main}
Let Assumptions {\bf (A1)}-{\bf (A3)} hold and suppose $d\ge Q$, where $Q$ is given in~\eqref{e:Q}. Let  
$F\subset \R^d$ be a bounded set that satisfies the following 
condition: There exist  constants $\theta \in[0, d- Q]$, $C_F \in (0, \infty)$,  and  $\kappa \in[0, (d-\theta)/Q)$ such that 
\begin{equation}\label{Con:F}
\lambda_d(F^{(r)})\le C_F r^{d-\theta} \Big( \log \log (1/r)\Big)^\kappa
\end{equation}
for all $r>0$ small, where $\lambda_d$ is the Lebesgue measure on $\R^d$ and 
\[
F^{(r)}=\Big\{x\in \R^d: \inf_{y\in F} |x-y|\le r\Big\}
\] 
is the (closed) $r$-neighborhood of $F$. 
Then $X^{-1}(F)\cap I=\emptyset$ a.s. 
\end{theorem}

Observe that (\ref{Con:F}) implies that the upper Minkowski (or box-counting) dimension of $F$ is at most 
$\theta$ (see, e.g., \cite[Proposition 2.4]{Falconer2014}) and is satisfied by many bounded sets $F$. 
The following corollary of Theorem \ref{Th:main} shows two cases that could not be handled by the 
hitting probability result in~\cite{BLX2009}.     
\begin{corollary} \label{Co:main}
Let Assumptions {\bf (A1)}-{\bf (A3)} hold and let  $F\subset \R^d$ be a bounded set.
\begin{itemize}
\item[(i)]\, If  $d > Q$, the Hausdorff dimension of $ F$ equals $d-Q$, and \eqref{Con:F} holds with $\theta = d-Q$ 
and a constant $\kappa < 1$, then  $X^{-1}(F)\cap I=\emptyset$ a.s. 
\item[(ii)]\, If $d = Q$ and $F$ satisfies  \eqref{Con:F}  with $\theta = 0$ and a constant $\kappa < 1$, then  
$X^{-1}(F)\cap I=\emptyset$ a.s. 
\end{itemize}
\end{corollary}

For any $t_o\in \R^N$ and constant $\eta > 0$, denote by $\bB_{\eta}(t_o)$ the closed ball in $\R^N$ centered  
at $t_o$ with radius $\eta$ in the metric $\Delta$, i.e.,
\begin{equation}\label{e:ball}
\bB_{\eta}(t_o)=\left\{ t\in \R^N: \Delta(t, t_o)\le \eta\right\}.
\end{equation}


For proving Theorem \ref{Th:main}, it suffices to show $X^{-1}(F) \cap \bB_{\eta}(t_o) =\emptyset$ a.s. 
for all $t_o\in I$, where $\eta>0$ is a small constant. Hence, we assume that $t_o \in I$ is fixed throughout 
the rest of this paper.  Let $t_o'$ be the corresponding point given in Assumption \textbf{(A3)} which is also fixed, 
and let $\eta$ be a fixed small positive number satisfying
\[
\mathbf B_\eta(t_o) \subset I^{(\epsilon_0)} \mbox{ and } \eta < \min \Big\{\frac12\min\{1, a_0^{-1}\}, {\rho_0}\Big\},
\]
 where we recall that the parameters $\epsilon_0, a_0$ are given in {\bf (A1)} and $\rho_0$ in {\bf (A3)}. 

For any $t\in \bB_{\eta}(t_o)$, denote 
\begin{equation}\label{eq:-def-X1X2}
	X^1(t)=X(t)-X^2(t), \quad
	X^2(t)=\E[X(t)|X(t_o')].
\end{equation}
Since $X$ is Gaussian, $X^1$ and $X^2$ are independent and  for $1\le j\le d$,
\begin{align} \label{e:X2}
	X_j^2(t) = \dfrac{\E \left[ X_j(t) X_j(t_o') \right]}{\E \left[X_j(t_o')^2\right]} X_j(t_o').
\end{align}
The following result, which is a generalization of \cite[Lemma 4.2]{Xiao97},  implies that $X^2$ has better 
path regularity than $X$ noting that $\delta_j>\alpha_j$ by Assumption {\bf(A3)}. Therefore, $X^1$ can 
be viewed as a small perturbation of $X$. 

\begin{lemma}\label{Lem:X2}
	Let Assumptions {\bf (A2)} and {\bf (A3)} hold. Then, for any $s,t\in \bB_{\eta}(t_o)$, 
	\begin{equation*}
		|X^2(s)-X^2(t)| \le K_1 |X(t_o')| \sum_{j=1}^N |s_j - t_j|^{\delta_j},
	\end{equation*}
	where $K_1>0$ is a finite constant only depending on $d_0$ and $t_o'$.
\end{lemma}

\begin{proof}
For $1 \le i \le N$, we have 
	\begin{align*}
		|X_i^2(s)-X_i^2(t)|
		=& \left| \dfrac{\E \left[ X_i(s) X_i(t_o') \right]}{\E \left[X_i(t_o')^2\right]} X_i(t_o') - \dfrac{\E \left[ X_i(t) X_i(t_o') \right]}
		{\E \left[X_i(t_o')^2\right]} X_i(t_o') \right| \\
		=& \dfrac{\Big| \E \left[ \left( X_i(s) - X_i(t) \right) X_i(t_o') \right] \Big|}{\E \left[X_i(t_o')^2\right]} \big| X_i(t_o') \big| \\
		\le& \dfrac{C}{d_0} \sum_{j=1}^N |s_j - \bar{s}_j|^{\delta_j} \big| X_i(t_o') \big|,
	\end{align*}
	where the first equality follows from \eqref{e:X2} and the inequality follows from  {\bf (A2)} and {\bf (A3)}. 
\end{proof}

We recall the following \cite[Proposition 2.3]{dalang2017polarity} which is analogous to \cite[Proposition 4.1]{Talagrand95} 
and is  the key ingredient for the construction of a random covering for the set $X^{-1}(F) \cap \bB_{\eta}(t_o)$.
\begin{proposition} \label{Prop:X}
Let Assumption {\bf (A1)} hold. Then there exist constants $K_2\in(0, \infty)$ and  $\delta_0\in(0,1]$ such that 
for any $r_0 \in (0, \delta_0)$ and $t \in I$, 
\[ \begin{split}
		&\bP \left\{\exists r \in [r_0^2, r_0], \sup_{s\in I^{(\epsilon_0)}:\Delta(s,t) < r} |X(s)-X(t)|
		\le K_2 r \left(\log \log \frac1r\right)^{- 1/Q} \right\}\\
		& \ge 1- \exp \left(-\sqrt {\log \frac1{r_0}} \right).
	\end{split}
	\]
\end{proposition}

The following lemma will also be used to construct the random covering in the sequel. More specifically, it will be used 
to control the size of the covering ball centered at  $X(t)$ where the local oscillation of $X$ around $t$ is larger than 
what is given in Proposition \ref{Prop:X}.

\begin{lemma} \label{Lem:X}
	Let Assumption {\bf (A1)} hold. Then there exists a constant $K_4\in(0,\infty)$ such that
	\begin{align}\label{e:X}
		\bP \Bigg( \sup_{s,t \in I^{(\epsilon_0)}: {\Delta(s,t)\le \varepsilon}} |X(s) - X(t)| 
		\le K_4 \varepsilon \sqrt{\log \dfrac{1}{\varepsilon}} \Bigg) \ge 1 - \varepsilon,
	\end{align}
	for all $\varepsilon \in(0,\frac12)$.
\end{lemma}

 Lemma \ref{Lem:X} follows from Lemma 2.1 in Talagrand
 \cite{Talagrand95} (there is a misprint in Lemma 2.1: on the right-hand side of (2.1), $D$ should be $D^2$). 
 For completeness, we provide a proof of Lemma \ref{Lem:X} by invoking a useful inequality presented in 
 \cite[Chapter 11]{LT91} for general stochastic processes.  
Recall that $\psi: \R_+ \to \R_+$ is called a  \emph{Young function} if it is convex and increasing such that $\psi(0)=0, \lim_{x\to \infty} 
\psi(x)=\infty$. The Orlicz space $L_\psi=L_\psi(\Omega, \mathcal A, \bP)$ associated to a Young function $\psi$ is the space of all 
real valued random variables $Y$ on $(\Omega, \mathcal A, \bP)$ such that $\E[\psi(|Y|/c)]<\infty$ for some $c>0$, and it is a 
Banach space under the norm 
\[\|Y\|_\psi=\inf\Big\{ c>0; ~\E[\psi(|Y|/c)]\le 1\Big\}. \]
In particular, it is easy to verify that if we choose $\psi(x)=e^{x^2}-1$, then for a centered Gaussian random variable $Y$, $\|Y\|_{\psi}
= C \|Y\|_{L^2}$ for some universal constant $C>0$. 

Let $T$ be an index set and $d$ be a pseudo-metric on $T$.  Consider a general stochastic process $Z=\{Z_t, t\in T\}$  
such that $\|Z_t\|_{\psi}<\infty$ for all $t\in T$ and 
\begin{equation}\label{condition-d}
\|Z_s-Z_t\|_\psi\le d(s,t), \text{ for all } s, t\in T.
\end{equation}
If we assume that the inverse function $\psi^{-1}$ of  the Young function $\psi$ satisfies
\begin{equation}\label{condition-psi}
\psi^{-1}(xy) \le C_{\psi}  \left(\psi^{-1}(x)+\psi^{-1}(y)\right)
\end{equation}
for some constant $C_\psi$ depending only on $\psi$,
then we have (see inequality (11.4) in \cite{LT91}), for all $u>0$,
\begin{equation}\label{e:ineq-LT}
\bP\left(\sup_{s,t\in T} |Z_s-Z_t|>8C_{\psi}\left(u+\int_0^D \psi^{-1} (N(T,d;\e)) d\e\right)\right)\le \Big( \psi(u/D) \Big)^{-1},
\end{equation}
where $N(T,d;\e)$ is the smallest number of open balls of radius $\e$ in the pseudo-metric $d$ which form a covering of $T$, 
and \[D=\sup_{s,t\in T} d(s,t)\] is the  diameter of  $T$ in the pseudo-metric $d$.

\begin{proof}[Proof of Lemma \ref{Lem:X}]
Denote the index set
\[
T_\e=\Big\{(t,\bar t) \in I^{(\epsilon_0)}\times I^{(\epsilon_0)}: \Delta (t, \bar t) \le \e\Big\}.
\]
To prove the desired result, we shall apply  \eqref{e:ineq-LT} with $\psi(x)=e^{x^2}-1$  to the Gaussian random field $Z$ on 
$T_\e$ defined by 
\[Z(\boldsymbol t)=Z(t, \bar t) =X(t)-X(\bar t), ~\boldsymbol t=(t, \bar t) \in T_\e.\]
Let $d_Z$ be the canonical metric on $T_\e$ induced by $Z$, i.e.,  for $\boldsymbol s=(s,\bar s)\in T_\e$ and $\boldsymbol t=(t, \bar t)
 \in T_\e,$ 
\begin{equation}\label{e:metric-d-Z}
d_Z(\boldsymbol s, \boldsymbol t) := \|Z(\boldsymbol s)-Z(\boldsymbol t)\|_{L^2} 
=\sqrt{\E\Big[\big(X(s)-X(\bar s)\big)-\big(X(t)-X(\bar t)\big) \Big]^2}.
\end{equation}
Then by the triangle inequality, we have
\begin{numcases}
{d_Z(\boldsymbol s, \boldsymbol t) \le} 
 \|X(s)-X(\bar s)\|_{L^2}+\|X(t)-X(\bar t)\|_{L^2} \label{case1}\\  \|X(s)-X(t)\|_{L^2}+\|X(\bar s)-X(\bar t)\|_{L^2} \label{case2}.
\end{numcases}
Thus by \eqref{case1} and Lemma \ref{Lem:distance},  the diameter $D$ of $T_\e$ in the metric $d_Z$ is at most $8 c_0 \e$. 

Next, for $\delta\in(0, \e)$, we count the number of balls of  radius $\delta$ in  $d_Z$ that are needed to cover $T_\e$.  
Recalling $Q=\sum_{j=1}^N\frac1{\alpha_j}$, we note that $T_\e$ in $\R^{2N}$ can be covered by, 
\begin{equation}\label{e:covering-number}
(8c_0N)^{2Q}\prod_{j=1}^N\frac{d_j-c_j}{\delta^{1/\alpha_j}} \prod_{j=1}^N \frac{\e}{\delta^{1/\alpha_j}}
\end{equation}
rectangles in $\R^{2N}$ of the form  $J_{1}\times J_{2}$, where $J_1, J_2\subset I^{(\epsilon_0)}$ are rectangles 
in $\R^N$ of the form $\prod_{j=1}^N [f_j, g_j]$ with $|f_j-g_j|^{\alpha_j}=\delta/(8c_0N)$. Thus, for all $\boldsymbol s= 
(s, \bar s), \, \boldsymbol t=(t, \bar t)\in J_1\times J_2$, we have by \eqref{case2} 
and Lemma \ref{Lem:distance},
\begin{align*}
d_Z(\boldsymbol s, \, \boldsymbol t)\le \|X(s)-X(t)\|_{L^2}+\|X(\bar s)-X(\bar t)\|_{L^2}\le  \delta,
\end{align*}
and  hence the diameters of such $J_1, J_2$ are not bigger than  $\delta$ in the metric $d_Z$.
Therefore, by \eqref{e:covering-number} we have 
\[N(T_\e, d_Z; \delta)\le C \left(\frac1 \delta\right)^{Q}\left(\frac\e\delta\right)^Q,\]
where $C$ is a constant depending only on $c_0, Q, N$ and $I$. 

Now, we apply \eqref{e:ineq-LT}. By choosing $\psi(x)=e^{x^2}-1$ and the metric $d$ defined by \eqref{e:metric-d-Z}, 
the conditions \eqref{condition-d} and \eqref{condition-psi} are satisfied, and we have for any $u>0$,
\begin{align}
	&\bP\left(\sup_{\boldsymbol s,\, \boldsymbol t\in T_\e} |Z(\boldsymbol s)-Z(\boldsymbol t)|>
	8C_{\psi}\left(u+\int_0^{8 c_0\e} \sqrt{\log (1+N(T_\varepsilon,d_Z;\delta))}\, d\delta \right)\right)\notag\\
	& \le \frac{1}{\exp(\frac{u^2}{64c_0^2\e^2})-1}\, .\label{e:X'}
\end{align}

Note that 
\[
\int_0^{8 c_0\e}\sqrt{\log N(T_\e, d_Z;\delta)} ~d\delta \le C \e \sqrt{\log\frac1\e}.
\]
Then by choosing  $\boldsymbol t=(t,t)$ and $u= K\e\sqrt{\log\frac1\e}$ for some proper positive constant $K$, the desired 
inequality \eqref{e:X} can be obtained by \eqref{e:X'}.  
\end{proof}

We introduce some notations before constructing a random covering for $X^{-1}(F)\cap \bB_{\eta}(t_o)$. For 
$k \in \bN_+$, define the random subset  $R_k$ of $\bB_{\eta}(t_o)$ as follows
\begin{equation}\label{eq-def-Rk}
	R_k = \left\{ t\in \bB_{\eta}(t_o): \exists r\in[2^{-2k}, 2^{-k}], \sup_{s\in I^{(\epsilon_0)}:\Delta(s,t) < r} |X(s)-X(t)|
	\le K_2 r \left(\log \log \frac1r\right)^{-1/Q} \right\}. 
\end{equation}

Proposition \ref{Prop:X} implies that for sufficiently large $k$, $\bP (t\in R_k)\ge 1-e^{-\sqrt{k}}$ for all $t\in I$. This and 
Fubini's theorem yield
\begin{equation} \label{eq-expected size of Rk}
\begin{split}
	\E[\lambda_N(R_k)]
	&= \E \Bigg[ \int_{\bB_{\eta}(t_o)} \1_{t \in R_k} \lambda_N(dt) \Bigg]
	= \int_{\bB_{\eta}(t_o)} \bP \left( t \in R_k \right) \lambda_N(dt)\\
	&\ge \lambda_N({\bB_{\eta}(t_o)}) \big(1-e^{-\sqrt{k}}\big),
\end{split}
\end{equation}
where $\lambda_N$ is the Lebesgue measure on $\R^N$. Also noting that $\lambda_N(R_k)\le \lambda_N({\bB_{\eta}(t_o)})$ 
a.s., by the Markov inequality and \eqref{eq-expected size of Rk}, one can derive that
\begin{align*}
	\bP \left(\lambda_N(R_k) < \lambda_N(\bB_{\eta}(t_o)) \big(1-e^{-\sqrt k/2}\big) \right)
	&= \bP \left(\lambda_N({\bB_{\eta}(t_o)}) - \lambda_N(R_k) > \lambda_N(\bB_{\eta}(t_o)) e^{-\sqrt k/2} \right) \\
	&\le \frac{ \E \big[ \lambda_N(\bB_{\eta}(t_o)) - \lambda_N(R_k) \big]} {\lambda_N({\bB_{\eta}(t_o)}) e^{-\sqrt k/2} }
	\le e^{-\sqrt k/2}.
\end{align*}
Thus, denoting
\begin{align*}
	\Omega_{k,1}:=
	\left\{\omega: \lambda_N(R_k) \ge \lambda_N(\bB_{\eta}(t_o)) \big( 1-e^{-\sqrt k/2} \big) \right\},
\end{align*}
we have
\begin{align} \label{eq-P Omega_{k1}}
	\sum_{k=1}^\infty \bP(\Omega_{k,1}^c)<\infty.
\end{align}

Recalling that $\delta_j > \alpha_j$ for $1 \le j \le N$, we can choose $\beta \in (0,1)$ such that $\beta < 
\delta_j \alpha_j^{-1} - 1$ for all $1 \le j \le N$.  Let
\begin{align*}
	\Omega_{k,2}=\big\{ \omega: |X(t_o')|\le 2^{k \beta} \big\}.
\end{align*}
Since $X(t_o')$ is a Gaussian random variable, we have
\begin{align} \label{eq-P Omega_{k2}}
	\sum_{k=1}^\infty \bP(\Omega_{k,2}^c)<\infty.
\end{align}
For $k \in \bN_+$, similar to $R_k$ given in  \eqref{eq-def-Rk}, we define the random set
\begin{equation}\label{eq-def-Rk'}
R_k'=\left\{ t\in \bB_{\eta}(t_o): \exists r \in [2^{-2k}, 2^{-k}], \sup_{s\in I^{(\epsilon_0)}:\Delta(s,t) < r} |X^1(s)-X^1(t)| 
	\le K_3\, r \left(\log \log \frac1r\right)^{-1/Q} \right\},
\end{equation}
where $K_3 = NK_1+K_2$. Note that on the event $\Omega_{k,2}$, by the triangle inequality and Lemma~\ref{Lem:X2}, 
we have for $t\in \bB_{\eta}(t_o)$ and $r\in[2^{-2k}, 2^{-k}]$ with $k$ being sufficiently large,
\begin{align} \label{eq-X to X1}
	&\sup_{s \in I^{(\epsilon_0)}:\Delta(s,t) < r} |X^1(s)-X^1(t)|\notag\\
	&\le \sup_{s  \in I^{(\epsilon_0)} :\Delta(s,t) < r} |X(s)-X(t)|
	+ \sup_{s  \in I^{(\epsilon_0)}:\Delta(s,t) < r} |X^2(s)-X^2(t)| \nonumber \\
	& \le  \sup_{s \in I^{(\epsilon_0)}:\Delta(s,t) < r} |X(s)-X(t)|
	+ K_1 |X(t_o')| \sup_{s \in I^{(\epsilon_0)}:\Delta(s,t) < r} \sum_{j=1}^N |s_j - t_j|^{\delta_j} \nonumber \\
	& = \sup_{s \in I^{(\epsilon_0)}:\Delta(s,t) < r} |X(s)-X(t)|
	+ K_1 |X(t_o')| \sup_{s \in I^{(\epsilon_0)}:\Delta(s,t) < r} \sum_{j=1}^N \big( |s_j - t_j|^{\alpha_j} \big)^{\delta_j \alpha_j^{-1}} \nonumber \\
	&\le \sup_{s \in I^{(\epsilon_0)}:\Delta(s,t) < r} |X(s)-X(t)|
	+ K_1 N 2^{k \beta} r^{\min_j\{\delta_j \alpha_j^{-1}\}} \nonumber \\
	&=  \sup_{s \in I^{(\epsilon_0)}:\Delta(s,t) < r} |X(s)-X(t)|
	+ K_1 N \big( 2^kr \big)^{\beta} r^{\min_j\{\delta_j \alpha_j^{-1}\} - \beta}.
\end{align}
Thus, by \eqref{eq-X to X1} and \eqref{eq-def-Rk}, noting $2^kr\le 1$ and $\min_j\{\delta_j\alpha_j^{-1}\}-\beta>1$, we have that  
for sufficiently large $k$, $R_k(\omega)\subset R_k'(\omega)$ for all $\omega\in \Omega_{k,2}$. Hence, $\Omega_{k,1} 
\cap \Omega_{k,2}\subset \Omega_{k,3}$ for sufficiently large $k$, where
\begin{align*}
	\Omega_{k,3}:= \left\{\omega: \lambda_N(R_k') 
	\ge \lambda_N(\bB_{\eta}(t_o)) \big(1-e^{-\sqrt k/2}\big)\right\}.
\end{align*}
Thus, by \eqref{eq-P Omega_{k1}} and \eqref{eq-P Omega_{k2}}, we have
\begin{align} \label{eq-P Omega_{k3}}
	\sum_{k=1}^\infty \bP(\Omega_{k,3}^c)
	\le \sum_{k=1}^\infty \bP(\Omega_{k,1}^c)
	+ \sum_{k=1}^\infty \bP(\Omega_{k,2}^c)
	<\infty.
\end{align}

The following lemma will be needed in the construction of a covering of the inverse image. It provides a nested family 
of subsets that shares similar properties with dyadic cubes in the Euclidean spaces, but is adapted to the anisotropic 
metric $\Delta$ in our setting. For ease of description, for every $q \in \bN_+$ we will call the sets $\{I_{q,l}\}$ in Lemma 
\ref{Lem:covering}  \emph{dyadic cubes of order $q$ in the metric $\Delta$}.

\begin{lemma} {\cite[Lemma 3.9]{Dalang2021}} \label{Lem:covering}
Let $T$ be a set in $\R^N$ equipped with the metric $\Delta$. There exist a constant $c_1 \in (0,1)$, a sequence 
$\{m_q: q \in \bN_+\}$ of positive numbers, and a family $\{I_{q,l}: 1 \le l \le m_q, q \in \bN_+\}$ of Borel subsets of $T$, 
such that
\begin{enumerate}
	\item[(i)] For all $q \in \bN_+$, $T = \bigcup_{l=1}^{m_q} I_{q,l}$.
	\item[(ii)] For $q_1 \ge q_2$, $1 \le l_1 \le m_{q_1}$, $1 \le l_2 \le m_{q_2}$, either $I_{q_1, l_1} \cap I_{q_2, l_2} 
	= \emptyset$ or $I_{q_1, l_1} \subset I_{q_2, l_2}$ holds.
	\item[(iii)] For each $q, l$, there exists $x_{q,l} \in T$ such that $\bB_{c_1 2^{-q-1}}(x_{q,l}) \subset I_{q,l} 
	\subset \bB_{2^{-q-1}}(x_{q,l})$ and $\{x_{q,l}: 1 \le l \le m_q\} \subset \{x_{q+1,l}: 1 \le l \le m_{q+1}\}$ for $q \in \bN_+$.
\end{enumerate}
\end{lemma}

\

Now we are ready to prove Theorem \ref{Th:main}.

\begin{proof}[Proof of  Theorem \ref{Th:main}]\, As mentioned earlier, it is sufficient to prove $X^{-1}(F) \cap \bB_{\eta}(t_o) 
=\emptyset$ a.s.,  where $t_o\in I$ is fixed and $\eta> 0$ is a small constant. To this end, we construct a random covering 
for $X^{-1}(F) \cap \bB_{\eta}(t_o)$ by modifying the approach used in \cite{Talagrand95, Xiao97,dalang2017polarity}. We 
choose $T = \bB_{\eta}(t_o)$ in Lemma \ref{Lem:covering}, then there exists a family $\{I_{q,l}: 1 \le l \le m_q, q \in \bN_+\}$ 
of dyadic cubes in the metric $\Delta$ such that $\{I_{q,l}: 1 \le l \le m_q\}$ forms a covering of $\bB_{\eta}(t_o)$. For every 
$t\in T$ and $n \ge 1$, let $C_n(t)$ be the unique dyadic cube of order $n$ which contains $t$. Then by (iii) of Lemma 
\ref{Lem:covering}, for all $u,v \in C_n(t)$,
\begin{align}\label{e:distance-Cn}
	\Delta(u,v) < 2^{-n}.
\end{align}
We call $C_n(t)$ a \emph{good dyadic cube} of order $n$ if
\begin{equation}\label{eq-good-Cl}
	\sup_{u, v\in C_n(t)} |X^1(u)-X^1(v)|
	\le 8 K_3 2^{-n}(\log\log 2^n)^{-1/Q}.
\end{equation}

By Definition \eqref{eq-def-Rk'} of $R_k'$, we see that for each $t \in R_k'$, there exists $r \in [2^{-2k}, 2^{-k}]$ such that
\begin{align} \label{eq-0.18}
	\sup_{s\in I^{(\epsilon_0)}:\Delta(s,t) < r} |X^1(s)-X^1(t)| \le K_3 r \left(\log \log \frac1r\right)^{-1/Q}.
\end{align}
Assume $2^{-n} \le  r < 2^{-n+1}$, and it is easy to verify that $k \le n \le 2k$. 
By the triangle inequality and \eqref{eq-0.18}, we have 
\begin{align*}
	&\sup_{u, v\in C_n(t)} |X^1(u)-X^1(v)|
	\le 2 \sup_{u \in C_n(t)} |X^1(u)-X^1(t)| \\
	&\le 2 \sup_{u: \Delta(u,t) < r} |X^1(u)-X^1(t)| 
	\le 2 K_3 r \left(\log \log \frac1r\right)^{-1/Q} \\
	&\le 4 K_3 2^{-n} \Big(\log \big( \log 2^n - \log 2 \big) \Big)^{-1/Q} \\
	&\le 8 K_3 2^{-n} \Big(\log \log 2^n \Big)^{-1/Q}.
\end{align*}
This implies that for $t\in R_k'$, $C_n(t)$ is a good dyadic cube of order $n$ for some $n \in [k , 2k]$. 

Denote by $V_n$ the union of good dyadic cubes of order $n$, and let $U_k= \bigcup \limits_{n=k}^{2k} V_n$. 
 Then clearly $R_k'\subset U_k$, and hence $\big(\bB_{\eta}(t_o)\backslash U_k \big) \cap R_k'=\emptyset$.
We also denote by $\cH_1(k)$ the family of dyadic cubes contained in $U_k$. Note that $\bB_{\eta}(t_o)\backslash 
U_k$ is contained in a union of dyadic cubes of order $2k$, none of which meets $R_k'$, and let $\cH_2(k)$ 
denote the smallest family of such dyadic cubes. Recalling the definition of the dyadic cube in Lemma \ref{Lem:covering}, 
the volume of the dyadic cube of order $2k$ is at least $C_{\text{Vol}} 2^{-2kQ}$, where $C_{\text{Vol}}$ is a positive 
constant that only depends on $c_1$, $N$, $\alpha_1, \ldots, \alpha_N$. As the event $\Omega_{k,3}$ occurs, 
$\lambda_N(\bB_{\eta}(t_o)\backslash U_k) \le \lambda_N(\bB_{\eta}(t_o))-\lambda_N(R_k')\le \lambda_N(\bB_{\eta}(t_o)) 
e^{-\sqrt k/2},$ and thus the number of cubes in $\cH_2(k)$ is at most
\begin{equation}\label{eq-upper bound-H2}
	C_1 e^{-\sqrt k/2} \lambda_N(\bB_{\eta}(t_o)) \prod_{j=1}^N  2^{2k \alpha_j^{-1}} 
	= C_1 2^{2k Q} e^{-\sqrt k/2} \lambda_N(\bB_{\eta}(t_o)).
\end{equation}
Here, $C_1$ is a  positive constant depending on $C_{\text{Vol}}$, $N$ and $\alpha_1, \ldots, \alpha_N$.

Denote 
\[
\cH(k)=\cH_1(k) \cup \cH_2(k).
\] 
Then $\cH(k)$ is a random family of dyadic cubes of order $n$ for $k \le n \le 2k$ and clearly it only depends 
on $\{X^1(t), t\in \bB_{\eta}(t_o)\}$.  Let
\begin{align*}
	\cH=\bigcup_{k=1}^\infty \cH(k).
\end{align*}
Then $\cH$ is also $\Sigma_1$-measurable, where $\Sigma_1$ is the $\sigma$-field 
generated by $\{X^1(t), t\in \bB_{\eta}(t_o)\}$.

Now for any $A \in \cH$, where $A$ is a dyadic cube of order $n$,  define
\begin{equation}\label{eq-def-r_A}
	r_A=
	\begin{cases}
		8 K_3 2^{-n}(\log\log 2^n)^{- 1/Q}, &\mbox{ if } A \in \cH_1(k), ~ k \le n \le 2k, \\
		\frac12 K_4 2^{-n} \sqrt{n}, & \mbox{ if } A \in \cH_2(k), ~ n=2k, 
	\end{cases}
\end{equation} 
where $K_4$ is the constant given by Lemma \ref{Lem:X}. 
For every $A\in \cH$, we pick a distinguished point $t_A$ in $A$. Let 
\begin{equation}\label{eq-def-Omega_A}
	\Omega_A=\{ d(X(t_A), F)\le 2r_A\},
\end{equation}
where $d(x,F) =\inf_{y\in F} |x-y|$. Denote 
\begin{align}\label{def:Fk}
	\cF(k)= \{A\in \cH(k):~ \Omega_A \text{ occurs} \}.
\end{align}

Define the event
\begin{align*}
	\Omega_{k,4}
	=\bigg\{ \omega: \mbox{ for every dyadic cube $C_k$ of order $k$}, \sup_{s,t\in C_k} |X(s)-X(t)| \le K_4 2^{-k}\sqrt{k} \bigg\}.
\end{align*}
For $s,t \in C_k$, $\Delta(s,t) \le 2^{-k}$  by \eqref{e:distance-Cn}. Then we have
\begin{align} \label{eq-P Omega_{k4}}
	\sum_{k=1}^\infty \bP(\Omega_{k,4}^c)
	&= \sum_{k=1}^\infty \bP \bigg(\omega: \exists \mbox{ dyadic cube  $C_k$ of order $k$},  \, 
	\sup_{s,t\in C_k} |X(s)-X(t)| > K_4 2^{-k}\sqrt{k} \bigg) \nonumber \\
	&\le \sum_{k=1}^\infty \bP \bigg(\omega: \sup_{s,t\in I^{(\epsilon_0)}:\Delta(s,t)\le 2^{-k}} |X(s)-X(t)| > K_4 2^{-k}\sqrt{k} \bigg) 
	< \infty,
\end{align}
where the last inequality follows from Lemma \ref{Lem:X}.

Now define 
\begin{equation}\label{def:Ok}
\Omega_k =\Omega_{k,2} \cap \Omega_{k,3} \cap \Omega_{2k,4}.
\end{equation}
Then by \eqref{eq-P Omega_{k2}},
 \eqref{eq-P Omega_{k3}} and \eqref{eq-P Omega_{k4}}, we have
\begin{align*}
	\sum_{k=1}^\infty \bP(\Omega_{k}^c)<\infty,
\end{align*}
which, together with the Borel-Cantelli lemma, implies
\begin{align*}
	\bP \Big(\liminf_{k\to\infty} \Omega_k \Big)=1.
\end{align*}

We make the following two claims:

{\bf Claim 1.}  For $k$ large enough, on the event $\Omega_k$ defined in (\ref{def:Ok}), $\cF(k)$ covers 
$X^{-1}(F)\cap \bB_{\eta}(t_o)$, recalling that $\cF(k)$ is given in \eqref{def:Fk}. That is, for $k$ large enough, 
$\cF(k)$ is a random covering of $X^{-1}(F)\cap \bB_{\eta}(t_o)$ on $\Omega_k$. 
\begin{proof}[Proof of Claim 1]
For any $t\in X^{-1}(F)\cap \bB_{\eta}(t_o)$, $t$ lies in a cube in $\cH_1(k)\cup \cH_2(k)$. If $t\in A\in \cH_2(k)$, it 
follows directly from the definitions of $\Omega_{2k,4}$ and $r_A$ that
\begin{align*}
	d(X(t_A), F) \le d(X(t_A), X(t))
	\le K_3 2^{-2k}\sqrt{2k} = 2r_A.
\end{align*}

If $t \in A \in \cH_1(k)$ for some good dyadic cube $A$ of order $n$ with $k\le n\le 2k$, then for $s,r\in A$, $|s_j-r_j|
\le 2^{-\alpha_j^{-1}n}$. Recalling that $\beta < \min_{1 \le j \le N} \{ \delta_j \alpha_j^{-1} \} - 1$, by the triangle inequality, 
the definition of good dyadic cubes and Lemma \ref{Lem:X2}, for $k$ large enough,
\begin{align*}
	d(X(t_A), F)
	&\le |X(t_A) - X(t)| 
	\le |X^1(t_A)-X^1(t)| + |X^2(t_A)-X^2(t)| \\
	&\le 8 K_4 2^{-n}(\log\log 2^n)^{- 1/Q}
	+ K_1 |X(t_o')| \sup_{s,r \in A} \sum_{j=1}^N |s_j-r_j|^{\delta_j} \\
	&\le  r_A+ K_1 2^{k \beta} \sum_{j=1}^N 2^{- \delta_j \alpha_j^{-1} n}
	\le r_A + K_1N 2^{k \beta} 2^{-(1+\beta+\beta') n} \\
	&\le r_A+K_1N 2^{-n} 2^{-\beta' n} 
	\le 2r_A,
\end{align*}
where $\beta' = \min_{1 \le j \le N} \{\alpha_j^{-1}\delta_j \} - 1 - \beta>0$. Thus, in both cases  $t\in A\in \cF(k)$.
\end{proof}

{\bf Claim 2.} If there exist  constants $\theta \in [0, d]$, $\kappa \ge 0$, and some positive constant $C_F$ depending on $F$ only, 
such that   $\lambda_d(F^{(r)})\le C_F r^{d-\theta} \big(\log\log\frac1 r \big)^\kappa$
for all $r>0$ small, 
then for any $A\in \cH=\bigcup_{k=1}^\infty \cH(k)$,
\begin{align}\label{Eq:claim2}
	\bP(\Omega_A|\Sigma_1) \le K_5 r_A^{d-\theta} \Big(\log\log\frac1 {r_A}\Big)^\kappa
\end{align}
for some finite constant $K_5$.  We remark that for each $A\in\cH$, there exists an integer $k$ such that $A\in \cH(k)$ and, in this case,
$\bP(\Omega_A|\Sigma_1)=\bP(A\in \cF(k)  |\Sigma_1)$ by the definition \eqref{def:Fk} of $\cF(k)$. 

\begin{proof}[Proof of Claim 2]
 For $t \in \bB_{\eta}(t_o)$ and $1 \le i \le N$, it follows from  \eqref{e:X2} that 
	\begin{align*}
		X_i^2(t) = \bigg( 1+ \dfrac{\E \left[ (X_i(t) - X_i(t_o')) X_i(t_o') \right]}{\E \left[X_i(t_o')^2\right]} \bigg) X_i(t_o')
		= g_i(t) X_i(t_o').
	\end{align*}
Note that Assumptions {\bf (A2)} and {\bf (A3)} guarantee that $g_i(t)$ is bounded away from $0$ uniformly in $t \in \bB_{\eta}(t_o)$ 
and $1 \le i \le N$. Besides, $X_i(t_o')$ is a normal random variable. Hence, the joint probability density function of 
$X^2(t)$ is uniformly bounded in $t\in T$.
	
Therefore, for each $A\in \cH$, by the independence of $X^1$ and $X^2$, 
	\begin{align*}
		\bP(\Omega_A|\Sigma_1)
		&= \bP \bigg( \inf_{y \in F} |X(t_A) - y| \le 2r_A \Big| \Sigma_1 \bigg) \\
		&= \bP \bigg( \inf_{y \in F - X^1(t_A)} |X^2(t_A) - y| \le 2r_A \Big| \Sigma_1 \bigg) \\
		&= \bP \Big( X^2(t_A) \in \big( F - X^1(t_A) \big)^{(2r_A)} \Big| \Sigma_1 \Big) \\
		&\le K_5 r_A^{d-\theta}  \Big(\log\log\frac1 {r_A}\Big)^\kappa,
	\end{align*}
for some finite constant $K_5$. In deriving the last inequality, we have used the facts that  the joint density 
of $X^2(t)$ is uniformly bounded in $t \in \bB_{\eta}(t_o)$ and 
$$\lambda_d\Big(\big(F - X^1(t_A))^{(2r_A)}\Big)=\lambda_d \Big(F^{(2r_A)}\Big)\le
2^{d-\theta}C_F r_A^{d-\theta} \Big(\log\log\frac1 {r_A}\Big)^\kappa.$$
This verifies (\ref{Eq:claim2}).
\end{proof}

Define the function $\phi$ on $(0, \infty)$ by
\begin{equation}\label{fun:phi}
\phi(s) = s^{Q-d+\theta} \left(\log\log\frac1s\right)^{\frac{d-\theta} Q  -\kappa}.
\end{equation}
Notice that, under the assumption $\theta \in [0,\, d- Q]$ and $\kappa \in [0, \,(d-\theta)/Q)$, we have  
$\lim\limits_{s\to 0^+}\phi(s) =\infty.$

For any $A \in \cH$,  denote $D(A)=2^{-n}$ if $A\in C_n$. Since $\Omega_{k,3}$ is $\Sigma_1$-measurable, for $k$ large enough, it follows from (\ref{Eq:claim2}) in 
{\bf Claim 2} that
\begin{equation}\label{Eq:Ek3}
\begin{split}
	& \E \Bigg[ \1_{\Omega_{k,3}} \sum_{A\in \cF(k)} \phi(D(A)) \Bigg]
	= \E \Bigg[ \E \bigg[ \1_{\Omega_{k,3}} \sum_{A\in \cF(k)} \phi(D(A)) \bigg| \Sigma_1 \bigg] \Bigg]\\
	&=  \E \Bigg[ \1_{\Omega_{k,3}} \sum_{A\in \cH(k)} \E \Big[ \1_{A \in \cF(k)} \big| \Sigma_1 \Big] \phi(D(A)) \Bigg]\\
	&\le  K_5 \E \Bigg[ \1_{\Omega_{k,3}} \sum_{A\in \cH(k)} r_A^{d-\theta} \Big(\log\log\frac1 {r_A}\Big)^\kappa \phi(D(A)) \Bigg]\\
	 &\le  K_5 \E \Bigg[ \1_{\Omega_{k,3}} \sum_{A\in \cH(k)} r_A^{d-\theta} \Big(\log\log\frac1 {r_A}\Big)^\kappa D(A)^{Q-d+\theta} 
	 \Big(\log\log \frac1{D(A)}\Big)^{\frac{d-\theta} Q - \kappa} \Bigg].
\end{split}
\end{equation}
Recalling the definition \eqref{eq-def-r_A} of $r_A$, one can write
\begin{align*}
	r_A=
	\begin{cases}
		8 K_3 D(A) \bigg( \log\log \dfrac{1}{D(A)} \bigg)^{- 1/Q}, &\mbox{ if } A \in \cH_1(k), \vspace{0.3cm}	\\	
		\frac12K_4 D(A) \left(\log\frac1{D(A)}\right)^{1/2}, & \mbox{ if } A \in \cH_2(k). 
	\end{cases}
\end{align*}
To deal with the sum $\sum_{A\in \cH(k)}$, we will split it into the sum $\sum_{A\in \cH_1(k)} + \sum_{A\in \cH_2(k)}$. 
In order to avoid duplication, $\sum_{A\in \cH_1(k)}$ only sums over all good dyadic cubes $A$ that are not included 
in another good dyadic cube, and this is where we use the nested property of the cubes given by Lemma \ref{Lem:covering}.
For every $A\in \cH_1(k)$, one can verify that 
\begin{equation}\label{Eq:Ek4}
r_A^{d-\theta} \Big(\log\log\frac1 {r_A}\Big)^\kappa D(A)^{Q-d+\theta} 
	 \Big(\log\log \frac1{D(A)}\Big)^{\frac{d-\theta} Q - \kappa} \le K_6\, D(A)^Q,
\end{equation}
where $K_6$ is a constant depending only on $K_4$, $d$, $\theta$, and $\kappa$. Besides, recalling the definition 
of the dyadic cube in Lemma \ref{Lem:covering}, the volume of the dyadic cube $A$ is at least $C_{\text{Vol}} D(A)^Q$. 
For  every $A\in \cH_2(k)$, one can verify that 
\begin{equation}\label{Eq:Ek5}
\begin{split}
& r_A^{d-\theta} \Big(\log\log\frac1 {r_A}\Big)^\kappa D(A)^{Q-d+\theta} 
	 \Big(\log\log \frac1{D(A)}\Big)^{\frac{d-\theta} Q - \kappa} \\
	 &\le K_7\, D(A)^Q \left( \log \frac1{D(A)} \right)^{(d-\theta)/2} 
	\left(\log \log \frac1{D(A)} \right)^{(d-\theta)/Q},
	\end{split}
\end{equation}
where $K_7$ is a constant depending only on $K_3$, $d$, $\theta$, and $\kappa$. 

It follows from (\ref{Eq:Ek3}), (\ref{Eq:Ek4}), and (\ref{Eq:Ek5}) that for sufficiently large $k$ we have

\begin{equation}\label{Eq:Ek6}
\begin{split}
	&\E\bigg[\1_{\Omega_{k,3}} \sum_{A\in \cF(k) }\phi(D(A))\bigg] \\
	&\le K_5 K_6 \E \bigg[ \sum_{A\in \cH_1(k)} D(A)^Q \bigg] \\
	&  \qquad +  K_5 K_7 \E \Bigg[ \1_{\Omega_{k,3}} \sum_{A\in \cH_2(k)} D(A)^Q \left( \log \frac1{D(A)} \right)^{(d-\theta)/2} 
	\left(\log \log \frac1{D(A)} \right)^{(d-\theta)/Q}\Bigg]\\
	&\le  K_8  \lambda_N\big(\bB_{\eta}(t_o)\big) 
	 +  K_5 K_7 \E\Bigg[ \1_{\Omega_{k,3}} \sum_{A\in \cH_2(k)} 2^{-2kQ} \left( 2k\right)^{(d-\theta)/2} \left( \log (2k) \right)^{(d-\theta)/Q} \Bigg]\\
	&\le  K_9 \lambda_N\left(\bB_{\eta}(t_o)\right),
\end{split}
\end{equation}
where $ K_9$ is a constant and the last inequality follows from the upper bound \eqref{eq-upper bound-H2} of the 
size of $\cH_2(k)$ on $\Omega_{k,3}$. 

Now, let $\bar \Omega=\liminf\limits_{k\to \infty} \Omega_k$ and notice that $\mathbb P(\bar\Omega)=1$. 
We consider the following quantity related to $X^{-1}(F) \cap \bB_{\eta}(t_o)$:
\[
\phi\hbox{-}m\big(X^{-1}(F) \cap \bB_{\eta}(t_o)\big) := \liminf_{k \to \infty} \sum_{A\in \cF(k) } \phi( D(A)).
\]
To put this quantity in perspective, we mention that, when $Q - d +\theta  > 0$ (we do not consider this case in the present paper), 
it follows from {\bf Claim 1} that 
$\phi\hbox{-}m \big(X^{-1}(F) \cap \bB_{\eta}(t_o)\big)$  gives an upper bound for the $\phi$-Hausdorff measure 
of $X^{-1}(F) \cap \bB_{\eta}(t_o)$.

By Fatou's lemma and (\ref{Eq:Ek6}), we have 
\begin{align*} 
	 \E\left[ \phi\text{-}m\left(X^{-1}(F) \cap \bB_{\eta}(t_o)\right)\right]
	&\le  \E \bigg[ \liminf_{k\to\infty} \sum_{A\in \cF(k)} \phi(D(A)) \1_{\Omega_{k}} \bigg]\\
	&\le \liminf_{k\to\infty} \E \bigg[ \sum_{A\in \cF(k)} \phi(D(A)) \1_{\Omega_{k,3}} \bigg]
	<\infty.
\end{align*}
Thus, we have shown that $\phi\text{-}m\big(X^{-1}(F) \cap \bB_{\eta}(t_o)\big) <\infty$  almost surely. Observe that, if 
$\theta\le d-Q$, we have $\lim\limits_{s\to 0^+}\phi(s) = \infty$. This together with the finiteness of $\phi\text{-}m
\big(X^{-1}(F) \cap \bB_{\eta}(t_o)\big)$ forces $X^{-1}(F) \cap \bB_{\eta}(t_o)$ to be an empty set a.s. This finishes 
the proof of  Theorem~\ref{Th:main}.
\end{proof}

\section{Examples of Gaussian random fields}\label{sec:examples}

Theorem \ref{Th:main} obtained in Section \ref{sec:hitting} is applicable to a broad class of Gaussian random fields, 
including multiparameter fractional Brownian motions, fractional Brownian sheets, and the solutions of systems of linear 
stochastic heat and wave equations. In this section, we verify that these examples satisfy Assumptions {\bf (A1)}-{\bf (A3)} 
imposed in Theorem \ref{Th:main}.

\subsection{Multiparameter fractional Brownian motions}

A multiparameter fractional Brownian motion (or fractional Brownian field) with Hurst parameter $H \in (0, 1)$ 
is a centered $\R^d$-valued Gaussian random field $X = \{ X(t), t \in \R^N \}$ with 
continuous sample paths and covariance given by
$$
\E[X_j(s)X_k(t)] = \delta_{j,k} \frac 1 2 \left(|s|^{2H} + |t|^{2H} - |s-t|^{2H}\right),
$$
where $|\cdot|$ is the Euclidean norm in $\R^N$ and $\delta_{j,k}$ is the Kronecker symbol.

Regarding the hitting probabilities of $X$, Testard \cite{Testard86} and Xiao \cite{Xiao99} have proved the following results:
$$
\mathcal{C}_{d-N/H}(F) > 0 \Rightarrow \mathbb{P} \{ X(I) \cap F \ne \emptyset \} > 0
\Rightarrow \mathcal{H}_{d-N/H}(F) > 0,
$$
where $\mathcal{C}_\alpha$ denotes the Bessel-Riesz capacity of order $\alpha$
and $\mathcal{H}_\alpha$ denotes the $\alpha$-dimensional Hausdorff measure.
Dalang, Mueller and Xiao \cite{dalang2017polarity} have also discussed the polarity of points and proved 
that $X$ does not hit points in the critical dimension $d = N/H$.

Recall that the fractional Brownian motion $X$ admits the following integral representation (see \cite{Talagrand98, dalang2017polarity}):
$$
X(t) = C \int_{\R^N} \frac{1-\cos(t \cdot \xi)}{|\xi|^{H+N/2}} M_1(d\xi)
+ C \int_{\R^N} \frac{\sin(t \cdot \xi)}{|\xi|^{H+N/2}} M_2(d\xi),
$$
where $M_1$ and $M_2$ are independent Gaussian white noises on $\R^N$ with Lebesgue control measure,
and $C$ is a suitable constant. With this representation, we can define
$$
W(A, t) = C \int_{|\xi|^H \in A} \frac{1-\cos(t \cdot \xi)}{|\xi|^{H+N/2}} M_1(d\xi)
+ C \int_{|\xi|^H \in A} \frac{\sin(t \cdot \xi)}{|\xi|^{H+N/2}} M_2(d\xi)
$$
for $A \in \mathcal{B}(\R_+)$ and $t \in \R^N$. In \cite{dalang2017polarity}, it is shown in the proof of Theorem 6.1
that our condition {\bf (A1)} is satisfied with $a_0 = 0$ and $\gamma_j = H^{-1} - 1$ for $j = 1, \dots, N$. For every $t \in \R^N\backslash\{0\}$, 
the control measure $\nu_t$ of $W(\cdot, t)$  is given by
\[
\nu_t(A) = 2C^2 \int_{|\xi|^H \in A} \big( 1-\cos(t \cdot \xi)\big) \frac{d\xi} {|\xi|^{2H+N}}.
\]
Also, on any compact rectangle $I \subset \R^N \setminus \{0\}$, 
{\bf (A2)} and {\bf (A3)} are satisfied with $\delta_j = 1$ for all $j$.
Therefore, our Theorem \ref{Th:main} applies to the fractional Brownian motion
with $Q = N/H$ and improves Theorem~6.1 of \cite{dalang2017polarity}.

\subsection{Fractional Brownian sheets}

A fractional Brownian sheet with Hurst parameters $H_1, \dots, H_N \in (0, 1)$
is a centered, continuous, $\R^d$-valued Gaussian random field $\{ X(t), t \in \R^N_+ \}$ 
with covariance
$$
\E[X_j(s) X_k(t)] = \delta_{j,k} \prod_{i=1}^N \frac 1 2 \left(s_i^{2H_i} + t_i^{2H_i} - |s_i-t_i|^{2H_i}\right).
$$
When $H_i = 1/2$ for all $i$, $X$ is the Brownian sheet.
In this case, the result of Khoshnevisan and Shi \cite{khoshnevisan1999brownian} 
provides a complete characterization for the polar sets:
$F \subset \R^d$ is polar if and only if $\mathcal{C}_{d-2N}(F) = 0$.
It has been an open problem whether this result extends to fractional Brownian sheets.

In \cite[Section 5.1]{Dalang2021}, it is shown that the fractional Brownian sheet $X$ has the following representation:
$$
X(t) = C \sum_{p \in \{0, 1\}^N} \int_{\R^N} \prod_{j=1}^N \frac{f_{p_j}(t_j\xi_j)}{|\xi_j|^{H_j+1/2}} M_p(d\xi),
$$
where $f_0(x) = 1 - \cos(x)$, $f_1(x) = \sin(x)$,
$M_p$, $p \in \{0, 1\}^N$, are i.i.d.~$\R^d$-valued Gaussian white noises
on $\R^N$, and $C$ is a suitable constant.

Let $I = \prod_{j=1}^N[c_j, d_j]$ be a compact rectangle, where $0 < c_j < d_j < \infty$ ($j = 1, \dots, N$).
Set
$$
W(A, t) = C \sum_{p \in \{0, 1\}^N} \int_{\max_j |\xi_j|^{H_j} \in A} \prod_{j=1}^N \frac{f_{p_j}(t_j\xi_j)}{|\xi_j|^{H_j+1/2}} M_p(d\xi).
$$
By Lemma 5.1 of \cite{Dalang2021}, our condition {\bf (A1)} is satisfied with $a_0 = 0$ and $\gamma_j = H_j^{-1} - 1$ for $j = 1, \dots, N$.
It is clear that {\bf (A2)} is satisfied with $d_0 = \prod_{j=1}^N c_j^{2H_j} > 0$.
Also, Lemma 5.2 of \cite{Dalang2021} implies that {\bf (A3)} is satisfied with
$\delta_j = \min\{2H_j, 1\}$ for $j=1, \dots, N$.
Therefore, our Theorem \ref{Th:main} and Corollary \ref{Co:main} apply to the fractional Brownian sheet
with $Q = \sum_{j=1}^N H_j^{-1}$.

\subsection{Systems of linear stochastic heat equations}

For systems of linear and nonlinear stochastic heat equations, upper and lower 
bounds for hitting probabilities have been obtained by Dalang, Khoshnevisan and 
Nualart \cite{dalang2007hitting, dalang2009hitting, dalang2013hitting}.
Those bounds allow us to determine the polarity of $F \subset \R^d$ in non-critical dimensions. 
For the solution of the linear stochastic heat equation \eqref{she} below,
Dalang, Mueller and Xiao \cite[Theorem 7.1]{dalang2017polarity} have proved that
points are polar in its critical dimension, which is $d = (4+2N)/(2-\beta)$, where $\beta \in (0, 2 \wedge N)$
is the constant in (\ref{cov-colored}).
We can now use our main theorem to extend the latter result for a class of non-singleton sets $F$.

Let $u(t, x) = (u_1(t, x), \dots, u_d(t, x))$ be the solution of the following system of linear stochastic 
heat equations on $\R_+ \times \R^N$:
\begin{align}\label{she}
\begin{cases}
\frac{\partial}{\partial t} u_j(t, x) = \Delta u_j(t, x) + \dot M_j(t, x),& j =1, \dots, d,\\
u_j(0, x) = 0.
\end{cases}
\end{align}
We assume that $\dot M_1, \dots, \dot M_d$ are i.i.d.~Gaussian noises that are 
white in time and spatially homogeneous with spatial covariance given by the Riesz kernel, i.e., formally,
\begin{equation}\label{cov-colored}
\E[\dot M_j(t, x) \dot M_j(s, y)] = \delta(t-s) |x-y|^{-\beta}, \quad 0 < \beta < 2 \wedge N.
\end{equation}
If $N = 1$, it is also possible to take $\dot M_1, \dots, \dot M_d$ to be i.i.d.~space-time white noises 
(and set $\beta = 1$ in this case).

Let $\tilde M(d\tau, d\xi)$ be a $\bC^d$-valued space-time white noise, i.e., 
$\operatorname{Re}{\tilde M}$ and $\operatorname{Im}{\tilde M}$ are independent space-time white noises.
Define the Gaussian random field $X = \{X(t, x), (t, x) \in \R_+ \times \R^N\}$ by
$$
X(t, x) = \mathrm{Re} \int_\R \int_{\R^N} e^{-i\xi \cdot x} \frac{e^{-i\tau t} - e^{-t|\xi|^2}}{|\xi|^2 - i\tau} \,
\frac{\tilde M(d\tau, d\xi)}{|\xi|^{(N-\beta)/2} }.
$$
In \cite[Section 7]{dalang2017polarity}, it is shown that $X$ has the same law as the solution 
$u = \{ u(t, x), (t, x) \in \R_+ \times \R^N \}$ of \eqref{she}.
Also, it is shown that, for any compact rectangle $I$ in $(0, \infty) \times \R^N$, by setting
$$
W(A, t, x) = \mathrm{Re} \iint_{|\tau|^{\alpha_1}\vee |\xi|^{\alpha_2} \in A} e^{-i\xi \cdot x} \frac{e^{-i\tau t} - 
e^{-t|\xi|^2}}{|\xi|^2 - i\tau}\, \frac{\tilde M(d\tau, d\xi)}{|\xi|^{(N-\beta)/2} },
$$
our condition {\bf (A1)} is satisfied with $\gamma_j = \alpha_j^{-1} -1$
for $j = 1, \dots, 1+N$, where
$$
\alpha_1 = \frac{2-\beta}{4} \quad \text{and} \quad
\alpha_2 = \dots = \alpha_{1+N} = \frac{2-\beta}{2},
$$
and {\bf (A2)} and {\bf (A3)} are satisfied with $\delta_j = 1$ for all $j = 1, \dots, 1+N$.
Therefore, our Theorem \ref{Th:main} is applicable to the solution $u$ of \eqref{she}
with $Q = (4+2N)/(2-\beta)$.

The theorem can also be applied to systems of linear stochastic heat equations
with non-constant coefficients. 
Let $v(t, x) = (v_1(t, x), \dots, v_d(t, x))$, $(t, x) \in \R_+ \times \R^N$, be the solution of 
\begin{equation}\label{she2}
\begin{cases}
\frac{\partial}{\partial t} v_j(t, x) = \Delta v_j(t, x) + \sigma_j(t, x) \dot M_j(t, x), & j = 1, \dots, d,\\
v_j(0, x) = 0,
\end{cases}
\end{equation}
where $\dot M_1, \dots, \dot M_d$ are Gaussian noises as in \eqref{she},
and for each $j = 1, \dots, N$, $\sigma_j: \R_+ \times \R^N \to \R$
is a non-random continuous function such that for all $T > 0$, there exist $0 < c_T < C_T < \infty$ such that 
$c_T \le \sigma_j(t, x) \le C_T$ for all $(t, x) \in [0, T] \times \R^N$.
Define the Gaussian random field $X = \{X(t, x) = (X_1(t, x), \dots, X_d(t, x)), (t, x) \in \R_+ \times \R^N\}$ by
$$
X_j(t, x) = \mathrm{Re} \int_{\R} \int_{\R^N} (\Phi_{t, x} \ast \widehat \sigma_j)(\tau, \xi)
\,\frac{\tilde M(d\tau, d\xi)}{|\xi|^{(N-\beta)/2} },
$$
where $\widehat \sigma_j$ is the Fourier transform of $\sigma_j$ in the variables $(t, x)$ and
$$
\Phi_{t, x}(\tau, \xi) = e^{-i\xi \cdot x} \frac{e^{-i\tau t} - e^{-t|\xi|^2}}{|\xi|^2 - i\tau}.
$$
In \cite[Section 8]{dalang2017polarity}, it is shown that
$X$ has the same law as the solution 
$v = \{ v(t, x), (t, x) \in \R_+ \times \R^N \}$ of \eqref{she2}, and,
in addition, if the functions $\sigma_j$ satisfy Assumption 8.1 in 
\cite{dalang2017polarity}, then our conditions {\bf (A1)}-{\bf (A3)} are satisfied on 
any compact rectangle $I \subset (0, \infty) \times \R^N$.
In this case, our Theorem \ref{Th:main} is applicable to the solution $v$ of \eqref{she2}.

\subsection{Systems of linear stochastic wave equations}

For a class of nonlinear hyperbolic SPDEs driven by space-time white noise, 
Dalang and Nualart \cite{dalang2004potential} have given a complete 
characterization for a set to be polar. For systems of linear and nonlinear stochastic wave 
equations driven by white noise or colored noise, the polarity of sets in non-critical dimensions 
have been studied by Dalang and Sanz-Sol\'e \cite{dalang2010criteria, dalang2015hitting}.
The polarity of points in the critical dimension for the solution of the linear stochastic 
wave equation \eqref{swe} below has been solved by Dalang et al.~\cite[Theorem 9.1]{dalang2017polarity}, 
and we can now improve their results for non-singleton sets.

Consider the solution $u(t, x) = (u_1(t, x), \dots, u_d(t, x))$ of the following system of 
linear stochastic wave equations on $\R_+ \times \R^N$:
\begin{align}\label{swe}
\begin{cases}
\frac{\partial^2}{\partial t^2} u_j(t, x) = \Delta u_j(t, x) + \dot M_j(t, x),& j =1, \dots, d,\\
u_j(0, x) = 0, \quad \frac{\partial}{\partial t} u_j(0, x) = 0,
\end{cases}
\end{align}
where $\dot M_1, \dots, \dot M_d$ are Gaussian noises as in \eqref{she}
with $N = 1 = \beta$, or $N \ge 2$ and $1 \le \beta < 2 \wedge N$.

Let $\tilde M(d\tau, d\xi)$ be a $\bC^d$-valued space-time white noise,
and $\{ X(t, x), (t, x) \in \R_+ \times \R^N\}$ be the Gaussian random field defined by
$$
X(t, x) = \mathrm{Re} \int_\R \int_{\R^N} \frac{e^{-i\xi \cdot x - i \tau t}}{2|\xi|}
\left( \frac{1-e^{it(\tau+|\xi|)}}{\tau+|\xi|} - \frac{1-e^{it(\tau-|\xi|)}}{\tau - |\xi|} \right)
\frac{\tilde M(d\tau, d\xi)}{|\xi|^{(N-\beta)/2} }.
$$
In \cite[Section 9]{dalang2017polarity}, it is shown that $X$ has the same law as the solution
$u = \{ u(t, x), (t, x) \in \R_+ \times \R^N \}$ of \eqref{swe}.
It is also shown that, for any compact rectangle $I$ in $(0, \infty) \times \R^N$, 
by setting
$$
W(A, t, x) = \mathrm{Re} \iint_{|\tau|^\alpha \vee |\xi|^\alpha \in A} \frac{e^{-i\xi \cdot x - i \tau t}}{2|\xi|}
\left( \frac{1-e^{it(\tau+|\xi|)}}{\tau+|\xi|} - \frac{1-e^{it(\tau-|\xi|)}}{\tau - |\xi|} \right)
\frac{\tilde M(d\tau, d\xi)}{|\xi|^{(N-\beta)/2} },
$$
our condition {\bf (A1)} is satisfied with $\gamma_j = \alpha^{-1} -1$ for $j = 1, \dots, 1+N$, where
$$
\alpha = \frac{2-\beta}{2},
$$
and {\bf (A2)} and {\bf (A3)} are satisfied with $\delta_j = 2-\beta$ for all $j = 1, \dots, 1+N$.
Therefore, our Theorem \ref{Th:main} is applicable to the solution $u$ of \eqref{swe}
with $Q = (2+2N)/(2-\beta)$.

\subsection{Rescaled Gaussian processes}
We prove the following Proposition \ref{Prop-transform} for rescaled Gaussian processes. As an 
example of application,  it implies that the Ornstein-Uhlenbeck process also satisfies {\bf (A1)}-{\bf (A3)}.

\begin{proposition} \label{Prop-transform}
Let $X$ be the Gaussian random field that satisfies {\bf(A1)}-{\bf(A3)} on $I^{(\epsilon_0)}$ for some positive 
constant $\epsilon_0$. Let $f > 0$ and $g_1, \ldots, g_N$ be locally Lipschitz continuous functions mapping $\R^N$ to $\R$. 
Denote $g(t) = (g_1(t_1), \ldots, g_N(t_N))$. Assume that there exist a compact interval 
$\tilde I$ and a positive constant $\tilde \epsilon_0$, such that $g(\tilde I^{(\tilde \epsilon_0)})$ is contained 
in $I^{\epsilon_0}$. Then the Gaussian random field $\widetilde{X}(t) := f(t) X(g(t))$ satisfies {\bf (A1)}-{\bf (A3)} 
on $\tilde I^{(\tilde \epsilon_0)}$.
\end{proposition}

\begin{proof}
Let $\{W(A,t):A \in \cB(\R_+), t \in \R^N\}$ be the Gaussian random field associated with $X$. Define the random field $\widetilde{W}$ by
\begin{align*}
	\widetilde{W}(A,t) = f(t) W(A,g(t)), \ \forall A \in \cB(\R_+), \, t \in \R^N.
\end{align*}
It is obvious that $\{\widetilde{W}(A,t):A \in \cB(\R_+), t \in \R^N\}$ is a Gaussian random field satisfying (a1) in the assumption {\bf (A1)}. 
Next, we verify condition (a2) in {\bf(A1)}. Let $\widetilde{a}_0 = 1 + a_0$, then for $s,t \in \tilde I^{(\tilde \epsilon_0)}$, $\widetilde{a}_0 \le 
a < b \le + \infty$, by the triangle inequality and the assumption that $X$ satisfies {\bf (A1)}, we have
\begin{align*}
	& \big\| \widetilde{W}([a,b),s) - \widetilde{X}(s) - \widetilde{W}([a,b),t) + \widetilde{X}(t) \big\|_{L^2} \nonumber \\
	=& \big\| f(s) W([a,b),g(s)) - f(s) X(g(s)) - f(t) W([a,b),g(t)) + f(t) X(g(t)) \big\|_{L^2} \nonumber \\
	\le& |f(s)| \big\| W([a,b),g(s)) - X(g(s)) - W([a,b),g(t)) + X(g(t)) \big\|_{L^2}  \nonumber \\
	&+ |f(s)-f(t)| \big\| W([a,b),g(t)) - X(g(t)) \big\|_{L^2} \nonumber \\
	\le& c_0|f(s)| \left[ \sum_{j=1}^N a^{\gamma_j} |g_j(s_j) - g_j(t_j)| + b^{-1} \right]
	+ |f(s)-f(t)| \big\| X(g(t)) \big\|_{L^2} \nonumber \\
	\le& c_0L \sup_{r \in I^{(\epsilon_0)}} |f(r)| \left[ \sum_{j=1}^N a^{\gamma_j} |s_j - t_j| + b^{-1} \right]
	+ L \sum_{j=1}^N |s_j - t_j| \sup_{t \in I^{(\epsilon_0)}} \big\| X(g(t)) \big\|_{L^2} \nonumber \\
	\le& C \left[ \sum_{j=1}^N a^{\gamma_j} |s_j - t_j| + b^{-1} \right],
\end{align*}
where we have used the Lipschitz continuity of $f, g_1, \ldots, g_N$, the boundedness of $\tilde I^{(\tilde \epsilon_0)}$, 
Lemma \ref{Lem:distance} and the fact that $a\ge \widetilde{a}_0 \ge 1$ in the last inequality. For the second inequality 
in (a2), by the triangle inequality, we have
\begin{align} \label{eq-a2-2}
	& \big\| \widetilde{W}([0,\widetilde{a}_0),s) - \widetilde{W}([0,\widetilde{a}_0),t) \big\|_{L^2} \nonumber \\
	=& \big\| f(s) W([0,\widetilde{a}_0),g(s)) - f(t) W([0,\widetilde{a}_0),g(t)) \big\|_{L^2} \nonumber \\
	\le& |f(s)| \big\|W([0,\widetilde{a}_0),g(s)) - W([0,\widetilde{a}_0),g(t)) \big\|_{L^2}
	+ |f(s)-f(t)| \big\| W([0,\widetilde{a}_0),g(t)) \big\|_{L^2}.
\end{align}
For the second term, by Lemma \ref{Lem:distance}, we have
\begin{align} \label{eq-a2-2-2nd}
	|f(s)-f(t)| \big\| W([0,\widetilde{a}_0),g(t)) \big\|_{L^2}
	\le L \sum_{j=1}^N |s_j - t_j| \sup_{t \in I^{(\epsilon_0)}} \big\| X(g(t)) \big\|_{L^2}
	\le C \sum_{j=1}^N |s_j - t_j|.
\end{align}
For the first term, since $X$ satisfies {\bf (A1)}, we have
\begin{align} \label{eq-a2-2-1st}
	& \big\|W([0,\widetilde{a}_0),g(s)) - W([0,\widetilde{a}_0),g(t)) \big\|_{L^2} \nonumber \\
	=& \big\| X(g(s)) - W([\widetilde{a}_0, \infty),g(s)) - X(g(t)) + W([\widetilde{a}_0, \infty),g(t)) \big\|_{L^2} \nonumber \\
	\le& c_0 \sum_{j=1}^N \widetilde{a}_0^{\gamma_j} |s_j - t_j|
	\le C \sum_{j=1}^N |s_j - t_j|.
\end{align}
The second inequality in (a2) is verified by substituting  \eqref{eq-a2-2-2nd} and \eqref{eq-a2-2-1st} to \eqref{eq-a2-2},

Noting that the continuity and the positivity of $f$ together with the compactness of $\overline{\tilde I^{(\tilde \epsilon_0)}}$ 
imply that $f$ is bounded away from $0$ on $\overline{\tilde I^{(\tilde \epsilon_0)}}$. Hence, $\widetilde{X}(t)$ satisfies {\bf (A2)}.

It remains to verify {\bf (A3)}. Noting that $f, g_1, \ldots, g_N$ are Lipschitz continuous, by the triangle inequality and the 
Cauchy-Schwarz inequality, we have
\begin{align*}
	& \Big| \E \left[ \widetilde{X}_i(t') \big( \widetilde{X}_i(s) - \widetilde{X}_i(\bar{s}) \big) \right] \Big| \nonumber \\
	=& |f(t')| \Big| \E \left[ X_i(g(t')) \big( f(s) X_i(g(s)) - f(\bar{s}) X_i(g(\bar{s})) \big) \right] \Big| \nonumber \\
	\le& |f(t')f(s)| \Big| \E \left[ X_i(g(t')) \big( X_i(g(s)) - X_i(g(\bar{s})) \big) \right] \Big|
	+ |f(t')| |f(s) - f(\bar{s})| \Big| \E \left[ X_i(g(t')) X_i(g(\bar{s})) \right] \Big| \nonumber \\
	\le& C \sup_{r \in I^{(\epsilon_0)}} |f(r)|^2 \sum_{j=1}^N \big| g_j(s_j) - g_j(\bar{s}_j) \big|^{\delta_j}
	+ \sup_{r \in I^{(\epsilon_0)}} |f(r)| L \sum_{j=1}^N |s_j - \bar{s}_j| \big\| X_i(g(t')) \big\|_{L^2} \big\| X_i(g(\bar{s})) \big\|_{L^2} \nonumber \\
	\le& C \sum_{j=1}^N |s_j - \bar{s}_j|^{\delta_j} + C \sum_{j=1}^N |s_j - \bar{s}_j|
	\le C \sum_{j=1}^N |s_j - \bar{s}_j|^{\delta_j},
\end{align*}
where the last inequality follows from $\delta_j \le 1$ for $1\le j\le N$. The proof is concluded. 
\end{proof}

As an example, we consider the Ornstein-Uhlenbeck process $X(t)$ defined by
\begin{align*}
	dX(t) = - \theta X(t) dt + \sigma dB(t),
\end{align*}
where $\theta$ and $\sigma$ are positive constants, and $B(t)$ is a 1-dimensional standard Brownian motion. 
It is well known that $X(t)$ can be represented as a time-space-rescaled Brownian motion, i.e., 
\begin{align*}
	\{X(t), t\in \R_+\} \overset d= \left\{\dfrac{\sigma}{\sqrt{2\theta}} e^{-\theta t} B(e^{2\theta t}), t\in\R_+\right\},
\end{align*}
where ``$\overset d=$'' means equality in distribution. Noting that $B(t)$ satisfies {\bf (A1)}-{\bf (A3)} on any compact 
interval on $\R\setminus \{0\}$ with $\gamma_1 = \delta_1 = 1$, one can show that the Ornstein-Uhlenbeck process $X(t)$ 
also satisfies {\bf (A1)}-{\bf (A3)}  with the same parameters by applying Proposition \ref{Prop-transform} with 
$f(t) = \frac{\sigma}{\sqrt{2\theta}} e^{-\theta t}$ and $g_1(t) = e^{2\theta t}$ .

\section{Collision of eigenvalues of random matrices} \label{sec:rm}

In this last section, we aim to apply our main result Theorem \ref{Th:main} to solve the problem on the 
collision of eigenvalues of random matrices that was left open in 
\cite{jaramillo2020collision, song2021collision}. 

Let $N \in \bN$ be fixed and consider a centered Gaussian random field $\xi = \{\xi(t): t \in \R_+^N \}$ defined on a probability 
space $(\Omega, \cF, \bP)$ with covariance given by
\begin{align*}
	\E \left[ \xi(s) \xi(t) \right] = C(s,t),
\end{align*}
for some non-negative definite function $C: \R_+^N \times \R_+^N \rightarrow \R$. Let $\{\xi_{i,j}, \eta_{i,j}:i,j \in \bN\}$ be a family of
 independent copies of $\xi$. For $\beta \in \{1,2\}$, and $d \in \bN$ with $d \ge 2$ fixed, consider the following $d\times d$ matrix-valued 
 process $X^{\beta} = \{X_{i,j}^{\beta}(t); t \in \R_+^N, 1 \le i,j \le d\}$ with entries given by
\begin{align} \label{def-entries}
	X_{i,j}^{\beta}(t) =
	\begin{cases}
	\xi_{i,j}(t) + \iota \1_{[\beta = 2]} \eta_{i,j}(t), & i < j; \\
	\sqrt{2} \xi_{i,i}(t), & i=j; \\
	\xi_{j,i}(t) - \iota \1_{[\beta = 2]} \eta_{j,i}(t), & i > j,
	\end{cases}
\end{align}
where $\iota := \sqrt{-1}$ is  the imaginary unit.  Clearly, for every $t \in  \R_+^N$, $X^{\beta}(t)$ is a real  symmetric matrix for 
$\beta=1$ and a  complex  Hermitian matrix for $\beta=2$.  In particular,  $X^1(t)/\sqrt{C(t,t)}$ belongs to  GOE 
and $X^2(t)/\sqrt{2C(t,t)}$ belongs to  GUE, respectively.

By the canonical identification, the matrix-valued process $X^{\beta}$ can be regarded as a Gaussian random field, still 
denoted by $X^{\beta}$, with values in  $\R^{d(d+1)/2}$ for $\beta = 1$ and in $\R^{d^2}$ for $\beta = 2$, respectively. The 
component processes of $X^\beta$ are independent, but are not identically distributed due to the constant factor of 
$\sqrt{2}$ in the diagonal entries. We denote by $D^\beta$ the invertible, diagonal matrix such that the Gaussian random field
$\widetilde{X}^\beta = D^\beta X^{\beta}$ has i.i.d. components. 

Let $A^1$ be a real symmetric deterministic matrix and $A^2$ be a complex Hermitian deterministic matrix. Suppose 
that $\{\lambda_1^{\beta}(t), \cdots,  \lambda_d^{\beta}(t)\}$ is the set of eigenvalues of 
\begin{align}\label{e:Y}
      Y^{\beta}(t) = A^{\beta} + X^{\beta}(t), \quad  (\beta=1,2).
\end{align}

Jaramillo and Nualart \cite{jaramillo2020collision} provided a necessary condition and a sufficient condition for the collision of 
eigenvalues of $Y^{\beta}$. The results were generalized by Song et al \cite{song2021collision} for the case where $k$ 
eigenvalues collide with $2\le k\le d$. More precisely, assuming that the associated Gaussian random field $\xi = \{\xi(t): t \in \R_+^N \}$ 
satisfies {(A1)} and {(A2)} in \cite{song2021collision}, we have, for the real case $\beta=1$:
\begin{enumerate}
	\item[(i)] if ~ $\sum_{j=1}^N\frac1{H_j} < (k+2)(k-1)/2$, then
	\begin{align}\label{Eq:nonC}
	\bP \left( \lambda_{i_1}^{\beta}(t) = \cdots = \lambda_{i_k}^{\beta}(t) \mathrm{\ for \ some \ } t \in I \ \mathrm{and} \ 
	1 \le i_1 < \cdots < i_k \le d \right) = 0;
	\end{align}
	\item[(ii)] if~ $\sum_{j=1}^N\frac1{H_j} > (k+2)(k-1)/2$, then
	\begin{align*}
		\bP \left( \lambda_{i_1}^{\beta}(t) = \cdots = \lambda_{i_k}^{\beta}(t) \mathrm{\ for \ some \ } t \in I \ \mathrm{and} \ 
		1 \le i_1 < \cdots < i_k \le d \right) > 0;
	\end{align*}
\end{enumerate}
for the complex case $\beta=2$:
\begin{enumerate}
	\item[(i)] if  $\sum_{j=1}^N\frac1{H_j} < k^2-1$, then
	\begin{align}\label{Eq:nonC'}
		\bP \left( \lambda_{i_1}^{\beta}(t) = \cdots = \lambda_{i_k}^{\beta}(t) \mathrm{\ for \ some \ } t \in I \ \mathrm{and} \ 
		1 \le i_1 < \cdots < i_k \le d \right) = 0;
	\end{align}
	\item[(ii)] if $\sum_{j=1}^N\frac1{H_j} > k^2-1$, then
	\begin{align*}
		\bP \left( \lambda_{i_1}^{\beta}(t) = \cdots = \lambda_{i_k}^{\beta}(t) \mathrm{\ for \ some \ } t \in I \ \mathrm{and} \ 
		1 \le i_1 < \cdots < i_k \le d \right) > 0.
	\end{align*}
\end{enumerate}
When $\sum_{j=1}^N\frac1{H_j} = (k+2)(k-1)/2$ for the real case and $\sum_{j=1}^N\frac1{H_j} = k^2-1$ for the complex case, 
the collision problems were left open by \cite{jaramillo2020collision} and \cite{song2021collision}.

Before studying the collision problem at the critical dimension, we first introduce some notations. We denote by $\mathbf{S}(d)$ 
and $\mathbf{H}(d)$ the set of real symmetric $d \times d$ matrices and the set of complex Hermitian $d \times d$ matrices, 
respectively. By the canonical identification, we have $\mathbf{S}(d) \simeq \R^{d(d+1)/2}$ and $\mathbf{H}(d) \simeq \R^{d^2}$. 
For $k \in \{1, \ldots, d\}$, let $\mathbf{S}(d;k)$ (resp. $\mathbf{H}(d;k)$) be the set of real symmetric (resp. complex Hermitian) 
$d \times d$ matrices with at least $k$ identical eigenvalues.

The following theorem solves the collision problem at the critical dimension for the real case $\beta=1$.

\begin{theorem} \label{Thm-hitting prob-real}
Let $Y^{\beta}$ ($\beta = 1$) be the matrix-valued process defined by (\ref{e:Y}) with eigenvalues $\{\lambda_1^\beta(t), \dots, 
\lambda_d^\beta(t)\}$. Assume the associated Gaussian random field $\xi = \{\xi(t): t \in \R_+^N \}$ satisfies {\bf (A1)}-{\bf (A3)}. For any 
$k\in\{2,\dots, d\}$, if~ $\sum_{j=1}^N\frac1{H_j} = (k+2)(k-1)/2$, then \eqref{Eq:nonC} holds.
\end{theorem}

\begin{proof}
It follows from Lemma 2.1 and Lemma 2.3 in \cite{song2021collision} that for any $M > 0$, the set 
$$
{\mathbf S}(d;k) \cap [-M, M]^{d(d+1)/2} \subseteq \mathrm{Im}(G)\cap [-M, M]^{d(d+1)/2}, $$
where $G : \R^{d+k-1}\times  \R^{\frac12[d(d-1) - k(k-1)]} \rightarrow {\mathbf S}(d)$ is a smooth function and the set 
$\mathrm{Im}(G)\cap [-M, M]^{d(d+1)/2}$ has positive and finite  $\big(\frac12[d(d+1) - k(k+1)]+1\big)$-dimensional 
Lebesgue measure. In the case of critical dimension (i.e., $\sum_{j=1}^N\frac1{H_j} = (k+2)(k-1)/2$),
$$\frac {d(d+1)} 2 -\sum_{j=1}^N \frac1{H_j} = \frac1 2 [d(d+1) - k(k+1)]+1,$$
we can verify that $(\mathrm{Im}(G) - A^{\beta}) \cap [-M, M]^{d(d+1)/2}$ and its image under the linear operator $D^\beta$ 
satisfy condition (\ref{Con:F}) of Theorem \ref{Th:main} with 
$\theta =\frac{ d(d+1)} 2 - Q$ and $\kappa = 0$. Applying Theorem \ref{Th:main} to the Gaussian random field $\widetilde{X}^{\beta}
= D^\beta X^\beta$
and $F = D^\beta \big(\big( \mathrm{Im}(G) - A^{\beta} \big)\cap  [-M, M]^{d(d+1)/2} \big)$, we obtain
\[
\begin{split}
&\bP \left( X^{\beta}(I) \cap \big( \mathrm{Im}(G) - A^{\beta} \big)\cap  [-M, M]^{d(d+1)/2} \not= \emptyset \right)\\
&= \bP \left( \widetilde{X}^{\beta}(I) \cap D^\beta\Big( \big(\mathrm{Im}(G) - A^{\beta} \big)\cap  [-M, M]^{d(d+1)/2}\Big) \not= \emptyset \right)= 0.
\end{split}
\]
Therefore
\begin{equation}\label{Eq:Hit-up}
\begin{split}
	& \bP \left( \lambda_{i_1}^{\beta}(t) = \cdots = \lambda_{i_k}^{\beta}(t) \mathrm{\ for \ some \ } t \in I \ \mathrm{and} \ 
	1 \le i_1 < \cdots < i_k \le d \right) \\
	&= \bP \left( Y^{\beta}(t) \in {\mathbf S}(d;k) \mathrm{\ for \ some \ } t \in I \right) \\
	&= \bP \left( X^{\beta}(t) \in \big( {\mathbf S}(d;k)  - A^{\beta} \big) \mathrm{\ for \ some \ } t \in I \right) \\
	&\le \bP \left( X^{\beta}(t) \in \big( \mathrm{Im}(G) - A^{\beta} \big) \mathrm{\ for \ some \ } t \in I \right) \\
	&= \bP \left( X^{\beta}(I) \cap \big( \mathrm{Im}(G) - A^{\beta} \big) \not= \emptyset \right)\\
	&= \lim_{M\to \infty} \bP \left( X^{\beta}(I) \cap \big( \mathrm{Im}(G) - A^{\beta} \big)\cap  [-M, M]^{d(d+1)/2} \not= \emptyset \right)= 0.
\end{split}
\end{equation}
This proves the non-existence of the $k$-collision of the eigenvalues.
\end{proof}

In Section \ref{sec:examples}, we have seen that the multiparameter fractional Brownian motions, fractional Brownian sheets,
 solutions to linear stochastic heat equations, and the Ornstein-Uhlenbeck processes satisfy the conditions {\bf (A1)}-{\bf (A3)}. 
 Hence, Theorem \ref{Thm-hitting prob-real} 
is applicable to these models.

\begin{corollary} \label{Coro-real}
Let $Y^{\beta}$ ($\beta = 1$) be the matrix-valued process defined by (\ref{e:Y}) with eigenvalues $\{\lambda_1^\beta(t), \dots, 	\lambda_d^\beta(t)\}$. The associated Gaussian random field $\xi = \{\xi(t): t \in \R_+^N \}$ is multiparameter fractional Brownian motion, fractional Brownian 
sheet, solution to linear stochastic heat equation, or the Ornstein-Uhlenbeck process. For any $k\in\{2,\dots, d\}$, if~ $\sum_{j=1}^N\frac1{H_j} 
= (k+2)(k-1)/2$, then \eqref{Eq:nonC} holds.
\end{corollary}

\begin{remark}
When $k=2$ and the associated Gaussian random field $\xi$ is (fractional) Brownian motion, Theorem \ref{Thm-hitting prob-real} 
recovers the non-collision property for the symmetric matrix Brownian motion (see e.g. \cite{anderson2010}).
Similarly, when $k=2$ and $\xi$ is  Ornstein-Uhlenbeck process,  Theorem \ref{Thm-hitting prob-real} recovers the non-collision 
property obtained in \cite{mckean} for the real symmetric matrix Ornstein-Uhlenbeck process.
\end{remark}

The following is the analogue of Theorem \ref{Thm-hitting prob-real} for the complex case $\beta=2$.
\begin{theorem} \label{Thm-hitting prob-complex}
Let $Y^{\beta}$ ($\beta = 2$) be the matrix-valued process defined by (\ref{e:Y}) with eigenvalues $\{\lambda_1^\beta(t), \dots, 
\lambda_d^\beta(t)\}$. Assume the associated Gaussian random field $\xi = \{\xi(t): t \in \R_+^N \}$ satisfies {\bf (A1)}-{\bf (A3)}. 
For any $k\in\{2,\dots, d\}$, if~ $\sum_{j=1}^N\frac1{H_j} = k^2-1$, then \eqref{Eq:nonC'} holds.
\end{theorem}

\begin{proof}
It follows from Lemma 3.2 and Lemma 3.4 in \cite{song2021collision} that 	for any $M > 0$, the set 
${\mathbf H}(d;k) \cap [-M, M]^{d^2} \subseteq \mathrm{Im}(\widehat{G})\cap [-M, M]^{d^2} $, which has finite 
$(d^2-k^2+1)$-dimensional Lebesgue measure. Here $\widehat{G}: \R^{d-k+1} \times \R^{d^2 - d - k^2 + k} \rightarrow {\mathbf H}(d)$
is a smooth function defined in  \cite[Lemma 3.2]{song2021collision}. 

In the case of critical dimension (i.e., $\sum_{j=1}^N\frac1{H_j} = k^2-1$),
$$d^2 -\sum_{j=1}^N \frac1{H_j} = d^2-k^2+1,$$
we see that $(\mathrm{Im}(\widehat{G}) - A^{\beta}) \cap [-M, M]^{d(d+1)/2}$ satisfies condition  (\ref{Con:F})  of Theorem \ref{Th:main} 
with $\theta = d^2 - Q$ and $\kappa = 0$. 
As in the proof of Theorem \ref{Thm-hitting prob-real}, We apply Theorem \ref{Th:main} to the Gaussian random field 
$\widetilde{X}^{\beta}$ to obtain
\[
\bP \left( X^{\beta}(I) \cap \big( \mathrm{Im}(\widehat{G} ) - A^{\beta} \big) \cap  [-M, M]^{d(d+1)/2} \not= \emptyset \right)= 0.
\]
Therefore
\begin{equation}\label{Eq:Hit-up'}
	\begin{split}
		& \bP \left( \lambda_{i_1}^{\beta}(t) = \cdots = \lambda_{i_k}^{\beta}(t) \mathrm{\ for \ some \ } t \in I \ \mathrm{and} \ 
		1 \le i_1 < \cdots < i_k \le d \right) \\
		&= \bP \left( Y^{\beta}(t) \in {\mathbf H}(d;k) \mathrm{\ for \ some \ } t \in I \right) \\
		&= \bP \left( X^{\beta}(t) \in \big( {\mathbf H}(d;k)  - A^{\beta} \big) \mathrm{\ for \ some \ } t \in I \right) \\
		&\le \bP \left( X^{\beta}(t) \in \big( \mathrm{Im}(\widehat{G}) - A^{\beta} \big) \mathrm{\ for \ some \ } t \in I \right) \\
		&= \bP \left( X^{\beta}(I) \cap \big( \mathrm{Im}(\widehat{G}) - A^{\beta} \big) \not= \emptyset \right)\\
		&= \lim_{M\to \infty} \bP \left( X^{\beta}(I) \cap  \big( \mathrm{Im}(\widehat{G}) - A^{\beta} \big)\cap [-M, M]^{d^2} 
		 \not= \emptyset \right)= 0.
	\end{split}
\end{equation}
This proves the non-existence of the $k$-collision of the eigenvalues.
\end{proof}

 Similar to the real case, we have the following result as a corollary of  Theorem
 \ref{Thm-hitting prob-complex}. 
\begin{corollary} \label{Coro-complex}
Let $Y^{\beta}$ ($\beta = 2$) be the matrix-valued process defined by (\ref{e:Y}) with eigenvalues $\{\lambda_1^\beta(t), \dots, \lambda_d^\beta(t)\}$. The associated Gaussian random field $\xi = \{\xi(t): t \in \R_+^N \}$ is  multiparameter fractional Brownian motion,
fractional Brownian sheet,  solution of linear stochastic heat equation, or Ornstein-Uhlenbeck process.  
For any $k\in\{2,\dots, d\}$, if~ $\sum_{j=1}^N\frac1{H_j} = k^2-1$, then \eqref{Eq:nonC'} holds.
\end{corollary}

{\bf Acknowledgements}   J. Song is partially supported by Shandong University (Grant No. 11140089963041) and 
the National Natural Science Foundation of China (Grant No. 12071256). Y. Xiao is supported in part by the NSF 
grant DMS-1855185.

\vspace{1cm}

\bibliographystyle{plain}
\bibliography{Reference.bib}

\end{document}